\documentclass[aip,cha,amsmath, amssymb,reprint,numerical]{revtex4-1}

\usepackage{amsthm}
\usepackage{graphicx}
\usepackage{dcolumn}
\usepackage{bm, bbm}
\usepackage{subfigure}
\usepackage{color}
\usepackage{ulem} 
\usepackage{verbatim}
\usepackage{epsfig}

\DeclareMathOperator{\im}{Im}
\DeclareMathOperator{\koo}{K_{\varphi}}
\DeclareMathOperator{\pfo}{P_{\varphi}}
\DeclareMathOperator{\fix}{\mathrm{fix}}
\DeclareMathOperator{\one}{\mathbbm{1}}
\DeclareMathOperator{\bone}{\mathbf{1}}

\newcommand{\overbar}[1]{\mkern 1.5mu\overline{\mkern-1.5mu#1\mkern-1.5mu}\mkern 1.5mu}
\newtheorem{mythm}{Theorem}[section]
\newtheorem{mylem}[mythm]{Lemma}
\newtheorem{mycor}[mythm]{Corollary}

\theoremstyle{definition}

\theoremstyle{remark}
\newtheorem*{myrem}{Remark}
\newtheorem{myex}{Example}[section]

\begin{document}


\title{Stochastic basins of attraction and generalized committor functions}%

\author{Michael Lindner}
\affiliation{Potsdam Institute for Climate Impact Research, Telegrafenberg A31, 14473 Potsdam, Germany}
\affiliation{Department of Mathematics, Humboldt University, Rudower Chaussee 25, 12489 Berlin,
Germany}

\author{Frank Hellmann}%
\affiliation{Potsdam Institute for Climate Impact Research, Telegrafenberg A31, 14473 Potsdam, Germany}

\date{\today}

\begin{abstract}

 We generalize the concept of basin of attraction of a stable state in order to facilitate the analysis of dynamical systems with noise and to assess stability properties of metastable states and long transients. To this end we examine the notions of mean sojourn times and absorption probabilities for Markov chains and study their relation to the basins of attraction. Our approach is directly applicable to all systems that can be approximated as Markov chains, including stochastic and deterministic differential equations. We also provide a sampling based generalization of basin stability that works without resorting to the Markov approximation by sampling trajectories directly.
\end{abstract}

\pacs{?}

\maketitle

\begin{quotation}

We discuss two far reaching generalizations of the basin of attraction of an attractor. These apply to general sets in the phase space of non-deterministic systems. The first is based on absorption probabilities, the second on expected mean sojourn times with respect to a finite time horizon. By casting the problem in the transfer operator language, we are able to give a simple formula for the first generalization along the lines of committor functions.

We show that the two notions of generalized basin coincide in the limit of a vanishing absorption probability and an infinite time horizon respectively. Importantly, for well-behaved deterministic systems this limit recovers the usual notion of basin of attraction. Finally we point out that derived quantities like the volume of the generalized basin are accessible through sampling trajectories at the same computational cost as evaluating basin stability for deterministic systems.

\end{quotation}

\section{Introduction}
 
There are two complementary approaches to the analysis of classical dynamical systems, focused either on individual trajectories generated by an iterated map or on observables or densities propagated by the associated transfer operators. Even if the underlying system is non-linear or stochastic, the associated transfer operators are always linear but usually act on infinite-dimensional spaces. 

Roughly speaking an attractor is a forward-invariant set $A$ which admits a larger set $V$, such that all states in $V$ converge to $A$. Its basin of attraction $B_A$ is the set of all such convergent states. An important quantity for many applications is the probability that a trajectory converges to a specific attractor after an initial, possibly large perturbation. Thus basin stability~\cite{menck2014dead, menck2013basin} of $A$ is defined as the volume of its basin of attraction under a given perturbation measure. An attractor $A$ with non-zero basin stability  is called Milnor attractor~\cite{milnor1985concept}.
So far basin stability has been studied almost exclusively in the context of deterministic dynamics, since if the system is stochastic the topological notions of attractor and basin of attraction are no longer meaningful and measure-theoretic notions have to be considered instead. 

An invariant set finds its analogon in the concept of an invariant measure. Ergodic theory~\cite{krengel1985ergodic, eisner2015operator} asks under which conditions and in which sense initial conditions converge to an invariant measure, if this measure is unique and what its characteristic properties are. In other words ergodic theory is concerned with the asymptotic behaviour of measure-preserving dynamical systems.

On finite time-scales metastable sets~\cite{bovier2002metastability} or almost-invariant sets~\cite{dellnitz1997almost, froyland2005statistically, froyland2009almost} correspond in many ways to attractors, that is they generalize an attractors transients properties, the influence it exerts on trajectories in its neighborhood.  When a system is a weak perturbation of a multistable deterministic system, the basins of attraction get perturbed into metastable sets in backward time. One can also consider sets which are metastable in both time directions, or the more general notion of sets that stay coherent under non-autnomous dynamics\cite{banisch2017understanding, froyland2010transport}.

Recently~\citet{serdukova2016stochastic} proposed the notion of stochastic basins of attraction for deterministic systems with noise. Their definition is based on studying the escape probability from the basin of attraction of the underlying deterministic system. In contrast, our definitions work without assuming an underlying deterministic system, or knowledge of its basin structure. This allows the construction of efficient estimators for the associated basin stability in high dimensional systems. We will further discuss the relationship to this work in the conclusion.

We will propose two closely related time-scale dependent notions of stochastic basins, based on mean sojourn times~\cite{rubino1989sojourn} and on hitting probabilities~\cite{norris1998markov}, which are also known as committor functions~\cite{metzner2009transition}. While these quantities can be defined for systems on a continuous state space as well, they originate in the theory of finite Markov chains, where their analysis is well understood. Commitor functions have been used previously as a tool to study deterministic basins, and to optimize basin stability in \cite{koltai2011efficient, koltai2011stochastic, koltai2014optimizing, neumann2004application}.

Since the focus of this paper is mainly in introducing new concepts we will often restrict our attention to Markov chains, thereby side-stepping technical difficulties that arise in continous state spaces or in continuous time. For many applications it suffices to develop the theory on the discrete level since the state spaces are naturally finite. If that is not the case, it is always possible to construct a finite rank approximation of the transfer operator which preserves the Markov property by a Galerkin scheme called Ulam's method~\cite{ulam1960collection, ding2002finite, klus2016numerical}. Computation of the Ulam matrix requires only knowledge of short trajectories and in some cases can be computed without integration of trajectories at all~\cite{froyland2013estimating}. The Perron-Frobenius theorem ensures the existence of a stationary distribution for the resulting Markov chain, however the convergence of these finite rank approximations to the stationary distribution of the transfer operator is a difficult open problem that has been solved only for some classes of systems~\cite{li1976finite, ding1996finite}.

\section{Theoretical foundations}

The following sections contain a brief introduction to dynamical systems and their operator-theoretic formulation. We will only touch upon some of the most important aspects of the theory and give references to literature. We will sometimes drop mathematical rigor for intuitive understanding and refer the interested reader to the literature and the appendix. The purpose of this chapter is to introduce Markov chains, and in particular mean sojourn times and committor functions, as powerfull tools for studying complicated dynamical behaviour.

\subsection{Discrete time dynamical systems \label{sec:topds}}

Let $(X,d)$ be a compact, metric space and let $\varphi: X \to X$ be a continuous map. The pair $(X,\varphi)$ is called discrete time dynamical system. For any fixed $x \in X$ its trajectory under $\varphi$ is the set $\{ \varphi^k(x) \mid k\in\mathbb{N} \}$, and its \emph{omega limit set} $\omega(x)$ is defined as
\begin{equation}
 \omega(x) = \bigcap_{K \in \mathbb{N}} \overbar{\{\varphi^k(x) \mid k \geq K\}},
 \label{eq:omegalimit}
\end{equation}
where $\overbar{A}$ is the closure of a set $A$. $\omega(x)$ is the set of all accumulation points of the trajectory of $x$.

$A \subseteq X$ is called \emph{backward-invariant} with respect to $\varphi$, if ${\varphi^{-1}(A):=\{x \in X \mid \varphi(x) \in A\} = A,}$ \emph{forward-invariant} if $\varphi(A)=A$ and \emph{absorbing} if $\varphi(A) \subseteq A$. It can be shown that $\omega(x)$ is closed, forward-invariant and non-empty for every $x \in X$.

For a set $A\subseteq X$ define its \emph{basin of attraction} $B_A$ as
\begin{equation}
 B_A = \{ x \in X \mid \omega(x) \subseteq A \}.
\end{equation}
 Milnor~\cite{milnor1985concept} calls the set $B_A$ realm of attraction to avoid confusion with other common definitions of basin of attraction. We state without proof that for any set $A$ its basin $B_A$ is backward-invariant under $\varphi$, and if $A$ is closed then $B_A$ is Borel-measurable.

\subsection{Stability}

As we have seen above the omega limit set $\omega(x)$ is the collection of asymptotic states of a trajectory starting at $x$. An important question for many applications as well as for numerical modelling is whether and in what sense limit sets, or, more broadly, invariant sets, are stable. Roughly speaking a set $A$ is locally stable if all states in a neighborhood converge to it. In other words small perturbations around $A$ do not change the asymptotic behaviour of the system. If one considers non-small perturbations that span the entire state space a global notion of stability is required. 

More formally, a set $A\subseteq X$ is \emph{Lyapunov stable} if for all neihbourhoods $B'$ of $A$ there exists another neihbourhood $B$ of $A$ from which trajectories end up in $B'$ eventually. That is, for every $x \in B$, there is a time $T$ after with which we have $\varphi^{T'}(x) \in B'$ for all $T' > T$.

$A$ is \emph{attractive} in a set $B \subseteq X$ if $\omega(x) \subseteq A$ for all $x \in B$. Further we say that $A$ is \emph{locally stable} if $A$ is Lyapunov stable and attractive in an open neighborhood $U$ of $A$ and that $A$ is \emph{globally stable} if $A$ is Lyapunov stable and attractive in $X$.

From the above definition it is clear that the basin of attraction $B_A$ of a set $A$ is the maximal subset of $X$ on which $A$ is attractive. For many applications it would be desirable to know $B_A$ exactly, in order to answer the question if a trajectory will converge to the same asymptotic state after an initial, possible large perturbation. Eventhough an exact characterization of $B_A$ is impossible in many cases of interest, for example due to the curse of dimensionality, it is often possible to compute the probability that a perturbed trajectory will converge back to $A$, as the volume of $B_A$ under a probability measure $\mu$ modelling the perturbation. \emph{Basin stability} of a closed set $A$ is then defined as the volume of its basin of attraction under $\mu$~\cite{menck2013basin}.

\begin{myrem} (Attractors)

 Naively speaking, attractors are subsets of state space to which some initial conditions converge asymptotically. Often, an attractor is conceptualized as an invariant set that fulfills some stability property, e.g. Lyapunov stability, and is minimal in the sense that it has no proper subset with the same properties. In his definition of attractors Milnor\cite{milnor1985concept} gives up the stability criterion and instead emphasizes `observability' by requiring that the corresponding basin of attraction has positive measure. According to Milnor a (minimal) attractor is a closed set $A \subseteq X$, such that
\begin{enumerate}
 \item $B_A$ has positive measure, with respect to a measure $\mu$ on the Borel $\sigma$-algebra of $X$;
 \item there is no strictly smaller closed subset $A^\prime \subset A$, such that $B_{A^\prime}$ has positive measure.
\end{enumerate}
In our terminology condition 1 states that $B_A$ should have positive basin stability. Since Milnor attractors need not even be Lyapunov stable, the term `stability' might be slightly misleading. Positive basin stability implies only a positive probability that $A$ is stable towards perturbations described by $\mu$.
\end{myrem}

\subsection{Transfer operators \label{sec:ctsop}}

So far we described a dynamical system by its trajectories, that is the action of a mapping $\varphi$ on states $x\in X$. Equivalently, we can study the action of the composition operator or \emph{Koopman operator} $\koo g = g \circ \varphi$ on observables $g$ in $L^\infty (X,\mu)$ for a given measure $\mu$. $\koo$ is a bounded and linear operator.

\begin{myrem}
 Since $X$ is compact, the space of continuous function from $X$ to $\mathbb{R}$ denoted by $C := C\left(X,\mathbb{R} \right) $ is contained in  $L^\infty (X,\mu)$. If  $g\in C$ is a fixed point of the restricted operator $\koo|_C$, that is $\koo|_C g = g $, then equivalently
\begin{equation}
  g(\varphi(x)) = g(x) \quad \forall x \in X.
\end{equation}
It follows that $g(\varphi^k(x)) = g(x)$ for all $k \in \mathbb{N}$ and hence $g$ is constant along trajectories and by continuity also on their limit sets and basins of attraction. We conclude that the fixed points of $\koo$ in $C$ characterize the basin structure of $\varphi$. In particular the constant function $\mathbbm{1}_X$ is a fixed point of $\koo$. For recent results on how to characterize the global stability of fixed points through the eigenfunctions of the Koopmann operator we refer the reader to ~\citet{mauroy2016global}.

\end{myrem}

The Koopman operator has a dual counterpart known as \emph{Perron-Frobenius operator} which generates the evolution of probability densities $f$ along trajectories. The Perron-Frobenius operator $\pfo$ acts on $L^1(X,\mu)$ and is defined by requiring that for all $\mu$-measurable sets $A \subseteq X$
\begin{equation}
 \int_A \pfo f \, \mu(dx) = \int_{\varphi^{-1}(A)} f \, \mu(dx).
\end{equation}
Some caution has to be taken to ensure that $\varphi$ and $\mu$ are compatible, for details we refer to \citet{lasota2013chaos}. If $f$ is a fixed point of $\pfo$, then $\mu_f(A):= \int_A f \mu(dx)$ is an invariant measure of $\varphi$, that is $\mu_f(\varphi^{-1}(A)) = \mu_f(A)$ for all measurable $A \subseteq X$.

When concerned with asymptotic behaviour invariant measures generalize in many ways the notion of an attractor. This property is formulated in the famous \emph{ergodic hypothesis} asking that for every (reasonable) observable $f:X\to\mathbb{R}$
\begin{equation}
 \lim_{N \to \infty} \frac{1}{N} \sum_{k=1}^{N} f(\varphi^k(x)) = \int_X f d\nu,
 \label{eq:ergo}
\end{equation}
Informally speaking, the ergodic hypothesis states that the time-average of $f$ along a trajectory should equal its space average with respect an invariant measure $\nu$. 

If we are given a measure $\mu$ that characterizes observable events, e.g. Lebesgue measure or a measure modelling a perturbation, then usually this measure is not preserved under $\varphi$. If it exists, the invariant measure such that Eq.~\eqref{eq:ergo} holds $\mu$-almost everywhere is called \emph{Sinai-Ruelle-Bowen (SRB) measure}~\cite{young2002srb} or sometimes \emph{physical measure}. Unfortunately, not every system has an SRB measure and it is an active area of research to find conditions on $\varphi$ that imply its existence\cite{catsigeras2011srb}. SRB measures are useful for similar reasons as Milnor attractors --- they ensure that the asymptotic states of the system are compatible with the given measure $\mu$.

\subsection{Markov operators and stochastic systems \label{sec:markovop}}

The Perron-Frobenius operator introduced in the previous section is a special case of a Markov operator. From now on we understand by a \emph{Markov operator}~\cite{lasota2013chaos} any linear operator $M:L^1(X,\mu) \to L^1(X,\mu)$ satisfying
\begin{enumerate}
 \item $Mf \geq 0$ for all $f \in L^1, f \geq 0$, we say M is \emph{positive}; and
 \item $ \lVert Mf \rVert_1 = \lVert f \rVert_1$ for all $f \in L^1, f \geq 0$, we say M is \emph{integral-preserving}.
\end{enumerate}

Just as the Perron-Frobenius operator describes the evolution of a density in the case of deterministic dynamics, a Markov operator describes 
the evolution of densities under stochastic dynamics. Markov operators are closely connected to Markov processes and their transition density functions~\cite{klus2016numerical}. For our purposes the Markov operator framework is convenient, since it allows to characterize the statistics of the process on a density level and avoids having to deal with individual trajectories.

Another convenient property of Markov operators is that they can be approximated by Markov operators of finite-rank~\cite{ding2009nonnegative}, which are just row-stochastic matrices. One such discretization scheme that goes back to an idea by Stanis\l{}aw Ulam is known as Ulam's method.
\subsection{Ulam's method for Markov operators \label{sec:ulam}}

Let $M:L^1(X,m) \to L^1(X,m)$ be a Markov operator, where $X \subset \mathbb{R}^n$ is compact and $m$ denotes the Lebesgue measure. Let $\mathcal{A}_h = \{A_1, \dots, A_K\}$ be a shape-regular partition of $X$ with mesh-size $h$. The basic idea of Ulam's method~\cite{klus2016numerical} is to obtain a coarse grained representation of the dynamics by considering only the flow of probability between partition elements. Consider the subspace $V_h \subset L^1$ spanned by the indicator functions $\one_{A_1},\dots,\one_{A_K}$. Let $Q_h: L^1 \to V_h$ be the projection onto $V_h$ given by
\begin{equation}
 Q_h f = \sum_{i=1}^K c_i \bone_i \quad \mathrm{with} \quad c_i = \int_{A_i} f dm,
\end{equation}
where $\bone_i = m(A_i)^{-1} \cdot \one_{A_i}$ denotes the $L^1$-normalized indicator functions.
\begin{mylem}
 The discretized operator $M_h := Q_h M|_{V_h}$ is a Markov operator as well.
\end{mylem}
\begin{proof}See \citet{ding2002finite}.
\end{proof}

 Denote the matrix representation of $M_h$ with respect to the basis $\bone_1,\dots,\bone_K$ by $\hat{M}_h$. Then the matrix entries $\hat{M}_{h,ij}$ are given by the relation
 
 \begin{equation}
  M_h \bone_i = \sum_{j=1}^K \int_{A_j} M \bone_i \ dm \cdot \bone_j = \sum_{j=1}^K \hat{M}_{h,ij} \bone_j,
 \end{equation}
and hence
 \begin{equation}
  \hat{M}_{h,ij} = \int_X \one_j M_h \bone_i dm = \int_X \one_j M \bone_i dm.
 \end{equation}

 \begin{myrem}
  We use the convention that matrices act by right-multiplication, i.e.
  \begin{equation}
   M_h (f) \equiv \hat{f}^T \hat{M}_h,
  \end{equation}
  where the left hand side signifies the operator acting on element $f \in V_h$ and the right hand side is the matrix representation acting on the vector representation $\hat{f}$ of $f$.
 \end{myrem}

 \begin{mycor}
  $\hat{M}_h$ is a (row-)stochastic matrix.
 \end{mycor}

 \begin{myex}
  Let $M_{\varphi}$ be a Perron-Frobenius operator with respect to the measurable map $\varphi:X \to X$.  In this case the matrix representation of $M_{ \varphi}$ is given by
  
\begin{align}
 \hat{M}_{h,ij} & =  \int_X \one_j M_{\varphi} \bone_i dm = \frac{1}{m(A_i)} \int_{\varphi^{-1}(A_j)} \one_{A_i} dm \nonumber \\
                & = \frac{m(A_i \cap \varphi^{-1}(A_j))}{m(A_i)}. 
 \label{eq:pfoulam}
\end{align}
 We see that the entries of the matrix equal the probability that a randomly chosen state in $A_i$ gets mapped to $A_j$ under action of $\varphi$. Therefore $\hat{M}_{h,ij}$ is often called \emph{transition matrix}.
 
In Section \ref{sec:ctsop} we stated that the Perron-Frobenius operator and Koopman operator are dual to each other. Duality holds as well for the discretized operators, such that the transposed transition matrix $M_h^T$ is an approximation of $\koo$, for details see e.g.~\citet{klus2016numerical}.
 \end{myex}

Ulam's method is a Galerkin projection~\cite{klus2016numerical} and was originally developed to approximate fixed points $f^*$ of the Perron-Frobenius operator $M_{\varphi}$. Thus an imporant question is when the fixed points $M_{\varphi,h}f_h = f_h $ of the finite-rank approximation converge to $f^*$ for appropriately refined partitions $\mathcal{A}_h$ of mesh size $h \to 0$. \citet{li1976finite}, and \citet{ding1996finite} proved convergence for certain classes of piecewise continuous maps on $\mathbb{R}^d$. Convergence of the fixed point equation in the presence of small random perturbation was shown in \citet{froyland1997computing, dellnitz1999approximation}. Indeed, the Galerkin discretization itself may be interpreted as such a small perturbation of $M$ that converges back to the full operator with increasing partition accuracy\cite{froyland1997computing}.

Since the transition matrix can be associated with a Markov chain it is sometimes referred to as Markov model~\cite{froyland2001extracting} and was found to characterize the system's dynamical properties, even in cases where convergence of the fixed densities cannot be shown. This approach proved useful as well for approximating eigenfunctions with eigenvalues close to 1, which characterize the metastable behaviour of the dynamical system~\cite{froyland2014well}. Motivated by its connection to probability flows, see Eq. \eqref{eq:pfoulam}, the transition matrix was recently interpreted as adjacency matrix of a weighted, directed graph, facilitating its analysis by tools from network theory~\cite{ser2015flow, lindner2017spatio}.

\subsection{Numerical implementation of Ulam's method}

We consider the case that $M := M_{\varphi}$ is a Perron-Frobenius operator. From now on denote the matrix representation $\hat{M_h}$ simply as $M_h$. The entries $M_{h,ij}$ can be interpreted as transition probabilities from box $i$ to box $j$ and are usually approximated by Monte-Carlo simulation. In every box $A_i$ a large number of test points $x_i^k$ with $k=1,\dots,K$ is randomly chosen, such that the transition probability can be estimated by the fraction of points that is mapped to box $A_j$,
\begin{equation}
 M_{h,ij} \approx \frac{1}{K}\sum_{k=1}^K \one_{A_j}(\varphi(x_i^k)).
\end{equation}
It is easy to check that the resulting matrix is still stochastic and thus a numerical realization of Ulam's method~\cite{froyland2001extracting, ding2002finite, klus2016numerical}.For non-deterministic system, such that we are able to simulate individual trajectories, the same approach can be applied. 

Having arrived at this point we leave the theory of transfer operators and focus on the less technical and more intuitive case of Markov chains. Aside from facilitating the analysis of hard-to grasp operators on functions spaces by approximation arguments, Markov chains are an indispensable tool in the modelling of real-world phenomena in their own right.

\section{Markov chains}

A \emph{Markov chain}\cite{norris1998markov} $(X_k)_{k \in \mathbb{N}}$ is a stochastic process on a discrete state space $\mathcal{X}$, such that $X_k$ is a random variable with values in $\mathcal{X}$ for all $k \in \mathbb{N}$ and such that the Markov property is satisfied, i.e. for all $k \in \mathbb{N}$:
\begin{equation}
 \mathbb{P}[X_{k+1} = i \mid X_k = i_k,\dots,X_0=i_0] =  \mathbb{P}[X_{k+1} = i \mid X_k = i_k]
\end{equation}
where $i, i_0, \dots, i_k$ are arbitrary elements of $\mathcal{X}$. A Markov chain is called \emph{homogeneous}  or \emph{stationary} if
\begin{equation}
 \mathbb{P}[X_{k+1} = i \mid X_k = j] = \mathbb{P}[X_{1} = i \mid X_0 = j] \quad \forall k \in \mathbb{N}.
\end{equation}

There is a one-to-one correspondence between homogeneous Markov chains on a finite state space $\mathcal{X}=\{1,\dots,n\}$ and stochastic matrices $M\in\mathbb{R}^{n\times n}$ by setting
\begin{equation}
 \mathbb{P}[X_{1} = j \mid X_0 = i] = M_{ij}.
\end{equation}
A \emph{probability distribution vector} $\rho \in \mathbb{R}^n$ is a non-negative vector, such that $\sum_{i=1}^{n}\rho_i = 1$. If $\rho := \rho(0)$ specifies the initial distribution of the Markov chain, i.e. $\mathbb{P}[X_0 = i] = \rho_i$, then the $k$-th step distribution $\rho(k)_i=\mathbb{P}[X_k = i \mid X_0 \sim \rho]$ can be computed as 
\begin{equation}
 \rho(k) = \rho(k-1)^T M = \rho^T M^k.
\end{equation}
From now on we assume that $(X_k)_{k \in \mathbb{N}}$ is homogeneous on a finite state-space.

\subsection{Sojourn times \label{sec:meansojourntime}}

Let $A \subseteq \{1,\dots,n\}$ be a subset of the state space. The \emph{mean sojourn time} in $A$ is the relative amount of time that the process spends in $A$. Let $\one_A$ denote the indicator function on $A$, then the mean sojourn time $\tau_s(A)$ along a trajectory is the random variable
\begin{equation}
 \tau_s(A) := \lim_{N \to \infty} \frac{1}{N} \sum_{k=0}^{N-1} \one_A(X_k).
\end{equation}
The \emph{expected mean sojourn time} or \emph{EMS time} in $A$ is
\begin{equation}
 s(A) := \mathbb{E}[\tau_s(A)] = \lim_{N \to \infty} \frac{1}{N} \sum_{k=0}^{N-1} \rho^T M^k \one_A,
 \label{eq:emst}
\end{equation}
where $\one_A \in \mathbb{R}^n$ is an indicator vector, that is $(\one_A)_i = 1$, if $i \in A$ and $0$ else, and $\rho$ is the initial distribution of the Markov chain. Equation~\eqref{eq:emst} holds since
\begin{align}
 \mathbb{E}[\one_A(X_k)] & = \mathbb{P}[X_k \in A] = \sum_{i \in A} \rho(k)_i = \rho^T M^k \one_A.
\end{align}
In Section \ref{sec:ergomat} we will see that
\begin{equation}
 \lim_{N \to \infty} \frac{1}{N} \sum_{k=0}^{N-1} M^k = P_{\fix(M)},
 \label{eq:sojournconvergence}
\end{equation}
where $P_{\fix(M)}$ is a projection onto $\fix(M) := \{ x \in \mathbb{R}^n \mid x^T M = x^T \}$. It follows that the EMS time in state $j$ given the process started in state $i$ is given by the entry $P_{\fix(M),ij}$ of the projection matrix.

\begin{myex} \label{ex:emstulamfix}
 If $M$ is the Ulam transition matrix of a Perron-Frobenius operator, then the fixed space of its transpose $\fix(M^T)$ is an approximation of the fixed space of the Koopman operator. In Section~\ref{sec:ctsop} we saw that if the underlying map $\varphi$ is continuous, then the eigenfunctions of the Koopman operator are constant along trajectories and in particular constant on basins of fix points. In this case the elements of $\fix(M^T)$ can be used to approximate the basin structure of $\varphi$ . 
\end{myex}

In the case that the Markov chain models a dynamical system after an initial perturbation, we are often interested in its asymptotics, that is whether, and to which equilibrium states the system returns. This question leads to Eq.~\eqref{eq:sojournconvergence}. On the other hand the transient behaviours might be important as well, that is in which way does the system return to the equilibrium states, how long does it take to do so and does it spend a long time in certain metastable states before reaching equilibrium. In order to study the transient behaviour, we will consider the finite-horizon EMS time
\begin{equation}
 s_N(A) := \frac{1}{N} \sum_{k=0}^{N-1} \mathbb{E}[\one_A(X_k)] = \frac{1}{N} \sum_{k=0}^{N-1} \rho^T M^k \one_A.
\end{equation}

\subsection{Committor functions \label{sec:classcomm}}

A closely related question is to determine the probability that the process ends up in a subset $A \subseteq \mathcal{X}$ given it started in state $i$. This \emph{absorption probability vector} $q$ can be obtained as the mininmal non-negative solution of the system of equations~\cite{norris1998markov}
\begin{equation}
 \begin{aligned}
 Mq & = q \quad \mathrm{ on } \ \mathcal{X}\setminus{A} \\ 
  q & = 1 \quad \mathrm{ on } \ A
 \label{eq:committorA}
\end{aligned}
\end{equation}

Similarly, for two disjoint sets $A$ and $B$, the probability  of not entering a set $B$ before having visited $A$ is given by the minimal non-negative solution of
\begin{align}
\nonumber
 Mq & =  q \quad \mathrm{ on } \ \mathcal{X}\setminus{(A\cup B)} \\
 \nonumber
  q & =  1 \quad \mathrm{ on } \ A                          \\
  q & =  0 \quad \mathrm{ on } \ B 
  \label{eq:committorAB}
\end{align}

The solution $q$ is called \emph{committor function}. An equivalent way of looking at the problem is to modify the process by adding two exit states $Z_A, Z_B$ to the state space, such that the transition probability from any state in $A$ to the exit state $Z_A$ is 1 and equally for $B$ and $Z_B$.

\section{New concepts for committor functions}

In this section we introduce two new generalizations of committor functions, fuzzy committors and $\varepsilon$-committors, and study some of their basic properties. The latter has close connections to basins of attraction as well as to finite-horizon EMS times (see also Appendix~\ref{sec:convergenceresults}) and provides insight into transient and asymptotic stability properties of systems that can be approximated by Markov chains	 (see also Appendix~\ref{sec:differenceestimates} and Section~\ref{sec:applications}).

\subsection{Fuzzy committors}
When we introduce exit states we are free to choose arbitrary transition probabilities to the exit states. Assume we have two transition probability distributions $p^1,p^2$ into exit states $Z_1, Z_2$. Additionally we require that $0 \leq p^1_i + p^2_i \leq 1$ and $ p^1_i, p^2_i \geq0$ for all $i\in\mathcal{X}$. To obtain the probability of being absorbed into $Z_1$ we introduce the matrix
\begin{equation}
\hat{Q}= 
\left(
\begin{array}{c|cc}
 \hat{M} & p^1 & p^2 \\
 \hline
 0       & 1   & 0   \\
 0       & 0   & 1
\end{array}
\right),
\end{equation}
where $\hat{M}_{ij} := M_{ij}\cdot(1-p^1_i-p^2_i)$ and obtain the minimal, non-negative solution $\hat{q}=(q,1,0)$ to the system
\begin{align}
 \hat{Q}\hat{q} & =  \hat{q} \quad \mathrm{ on } \ \mathcal{X} \nonumber \\
 \hat{q} & = 1 \quad \mathrm{ on } \ Z_1                \nonumber \\
 \hat{q} & =  0 \quad \mathrm{ on } \ Z_2.
\end{align}
This can be rephrased as
\begin{align}
 \nonumber \hat{M}q+p^1 & =  q\quad \\ 
  \Leftrightarrow  (I-\hat{M})q & =   p^1,
 \label{eq:fuzzycommittor}
\end{align}
where we denote the absorption probability into exit state $Z_1$ by $q$. Note however that $\hat{M} := \hat{M}(p^1,p^2) := \hat{M}\cdot\mathrm{diag}(\mathbbm{1}-p^1-p^2)$ depends on $p^1,p^2$. If we set $p^1=\mathbbm{1}_A,p^2=0$ we obtain Eq.~\eqref{eq:committorA} and if we set $p^1=\mathbbm{1}_A,p^2=\mathbbm{1}_B$ we obtain Eq.~\eqref{eq:committorAB}.

The probability distribution $q$ will be referred to as \emph{fuzzy committor with respect to $p^1$ and $p^2$}. The term \emph{fuzzy} refers to the notion of fuzzy sets described by affiliation functions like $p^1$ and $p^2$, generalizing the idea of `crisp' sets commonly described by binary indicator functions.

\subsection{$\varepsilon$-committors}

Within this paper we aim to address stability questions at finite time-scales, hence asymptotic absorption probabilities are not our main interest. If the right exit probabilites are chosen the fuzzy committors turn out as a convenient tool for studying transient behaviour. 

Consider a Markov chain that has probabiliy $\varepsilon$ of being absorbed into a unique exit state at every timestep and uniformly on its state space. Define a random variable $T_\varepsilon$ as the timestep when a trajectory is absorbed into the exit state. Thus $T_\varepsilon - 1$ is the time the system spends in the original state space. Clearly $\mathbb{P}[T_\varepsilon = 0] = 0$ and the probability that a trajectory hits the exit state at timestep $k > 0$ is
\begin{equation}
 \mathbb{P}[T_\varepsilon = k] = \varepsilon (1-\varepsilon)^{k-1} \quad \forall k \geq 1.
\end{equation}
The expected value of $T_\varepsilon$ is
\begin{equation}
 \mathbb{E}[T_\varepsilon]= \sum_{k=0}^{\infty} k \, \mathbb{P}[T_\varepsilon = k] = \varepsilon \sum_{k=1}^{\infty}  k (1-\varepsilon)^{k-1} = \dots = \frac{1}{\varepsilon},
\end{equation}

and hence the inverse of the exit probability $\varepsilon$ can be considered as the expected time-horizon.

We want to know the probability that a given trajectory starting in state $i$ spends a `long' time in $A$ with respect to a finite time-scale. In Section~\ref{sec:meansojourntime} we saw that the expected mean sojourn times offer a way to answer this question. An alternative approach is to choose an exit probability $\varepsilon \in (0,1]$ as the inverse of the time-scale of interest and to define $p^1 = \varepsilon \mathbbm{1}_A, \; p^2 = \varepsilon (\one - \mathbbm{1}_{A})$ as exit probabilites for the fuzzy committors. Then Eq.~\eqref{eq:fuzzycommittor} becomes 
\begin{equation}
 (I-(1-\varepsilon)M) q_\varepsilon(A)= \varepsilon \mathbbm{1}_A
 \label{eq:epscommittor}
\end{equation}
The solution $ q_\varepsilon(A) =: q_\varepsilon$ will be referred to as $\varepsilon$\emph{-committor} of $A$, where we drop the argument $A$ if it is clear from context. Existence and uniqueness of $ q_\varepsilon$ for $\varepsilon > 0$ follow by applying the Neumann inversion formula
\begin{equation}
 q_\varepsilon(A) = \varepsilon \left( \sum_{k=0}^{\infty} (1-\varepsilon)^k M^k \right) \mathbbm{1}_A.
 \label{eq:neumanncommittor}
\end{equation}

\begin{myrem}
We can replace the target set $A$ characterized by an indicator function $\one_A$ with a generalized state described by any vector $v \in \mathbb{R}^n$, such that $\max_i \lvert v_i \rvert = 1$. The corresponding $\varepsilon$-committor is denoted as $q_\varepsilon(v)$. In this case some care has to be taken when interpreting the results.
\end{myrem}

\subsection{Properties of $\varepsilon$-committors \label{sec:furthercommis}}

Let $(Y_k)_{k \in \mathbb{N}}$ be the original process $(X_k)_{k \in \mathbb{N}}$ modified in such a way that at each step a transition to one of the exit states occurs with probability $\varepsilon$.

\begin{mylem}
The $i$-th component  of the $\varepsilon$-committor $q_{\varepsilon,i}$ is the probability that a trajectory starting in state $i$ is in set $A$ just before it exits the system. 
\end{mylem}
\begin{proof}
Note that $T_\varepsilon - 1$ is the time-step before a given trajectory transitions to one of the exit states. Then
\begin{eqnarray}
\mathbb{P}[Y_{T_\varepsilon - 1}  \in  A\mid Y_0 & & = i]  = \sum_{k=0}^\infty\mathbb{P}[T_\varepsilon = k + 1]\mathbb{P}_i[Y_k\in A] \nonumber \\ 
& &=   \varepsilon \sum_{k=0}^{\infty} (1-\varepsilon)^k e_i^T M^k \mathbbm{1}_A = q_{\varepsilon,i}.
\end{eqnarray}
For the first equality we used that absorption is independent of the current state of the system.
\end{proof}

The $\varepsilon$-committor is closely related to the expected time a trajectory spends in a given set $A$. To see this, define a random variable $\tau_\varepsilon = \tau_\varepsilon(A)$ as the total time the modified process spends in the set $A$, that is
\begin{equation}
 \tau_\varepsilon = \sum_{k=0}^{\infty} \one_A(Y_k).
\end{equation}
Then the expected value of $\tau_\varepsilon$ is
\begin{align} \label{eq:expectedZ}
 \mathbb{E}[\tau_\varepsilon] & = \mathbb{E}\left[\sum_{k=0}^{\infty} \one_A(Y_k)\right] = \sum_{k=0}^{\infty} \mathbb{E}[ \one_A(Y_k)] = \sum_{k=0}^{\infty} \mathbb{P}[ Y_k \in A] \nonumber \\
 & = \sum_{k=0}^{\infty} \mathbb{P}[ X_k \in A]\cdot (1-\varepsilon)^k.
\end{align}
If we condition on an initial state $Y_0 = X_0 = i$, then Eq.~\eqref{eq:expectedZ} becomes
\begin{align}
 \mathbb{E}[\tau_\varepsilon\mid X_0=i] & = \sum_{k=0}^{\infty} \mathbb{P}_i[ X_k \in A]\cdot (1-\varepsilon)^k \nonumber \\
 &= \sum_{k=0}^{\infty} e_i^T M^k \one_A (1-\varepsilon)^k = \frac{1}{\varepsilon} \, q_{\varepsilon,i},
\end{align}
which relates the $\varepsilon$-committor to the expected time that the process spends in $A$ before absorption. 
This could be small because the process does not reach $A$, and thus $A$ is not attractive on this time scale, or because it quickly leaves $A$ again, that is, $A$ is not stable. 

If the limit $\varepsilon \to 0$ of $\frac{1}{\varepsilon} q_{\varepsilon,i}$ exists, it equals the expected time the original process $X_k$ spends in $A$. The probability distribution function of $\tau_\varepsilon(A)$ for general absorbing Markov chains is derived in~\citet{csenki2012dependability}, Corollary 2.8.

\begin{myex} (Naive model of long transients I)

 Assume that the proccess has a non-zero probability $p$ of leaving the set $A$ at every timestep, that is ${\mathbb{P}_i[X_k \in A] = (1-p)^k}$. It follows that
 \begin{equation}
  q_{\varepsilon,i} = \frac{\varepsilon}{1-(1-\varepsilon)(1-p)},
 \end{equation}
which can be solved for $p$ in order to obtain the leak rate as a function of the $\varepsilon$-committor.
\end{myex}

\begin{myex} (Naive model of long transients II) \label{ex:longtransientsiii}
 
 More generally speaking, a long transient set has the property that  $f(k):=\mathbb{P}_i[X_k \in A]$ tends to 0 with increasing $k$. For simplicity assume that $f(k)$ is monotonically decreasing and that $\varepsilon = 1/N$ for $N \in \mathbb{N}$. Then we have the following inequality between $s_N$, the EMS time with horizon $N$ and the $\varepsilon$-committor.
 
\begin{align}
 q_{\varepsilon,i} & =  \varepsilon \sum_{k=0}^{\infty} f(k) (1-\varepsilon)^k \nonumber \\
  &\leq \varepsilon \sum_{k=0}^{N-1} (f(k)-f(N)) (1-\varepsilon)^k + \varepsilon f(N) \sum_{k=0}^{\infty} (1-\varepsilon)^k \nonumber \\
  &\leq \frac{1}{N}\sum_{k=0}^{N-1} (f(k)-f(N)) + f(N) \nonumber \\
  &\leq \frac{1}{N} \sum_{k=0}^{N-1} \mathbb{P}_i[X_k \in A] = s_{N,i}.
\end{align}
Thus, in this case, the $\varepsilon$-committor is a lower bound for the EMS time with horizon $\varepsilon^{-1}$. If $f$ has an initial period where it increases before it eventually decreases monotonically, then the same reasoning applies if $N$ is chosen large enough. 
\end{myex}

The $\varepsilon$-committors and EMS times give in many ways similar information on the system's transient behaviour. The former performs a geometric averaging along the whole trajectory, while the latter just takes the mean value with respect to a finite time interval. In Section \ref{sec:differenceestimates} we obtain exact difference estimates between both quantities for specific metastable states. For the asymptotic case, that is if the (expected) time-horizon tends to infinity, we show in the next Section that both quantities have the same limit, and that this limit equals the indicator function on the basin of attraction if the underlying system is deterministic, compare Example~\ref{ex:emstulamfix}. This limit behaviour suggests that both quantities can be used as the sought generalizations of the concept of basin of attraction to stochastic systems and metastable phenomena.
Motivated by these considerations we define $\varepsilon$\emph{-absorption stability} $b_\varepsilon(A)$ of a set $A$ with respect to a probability measure $\rho \in \mathbb{R}^n$ as 
\begin{equation} \label{eq:epsabs}
 b_\varepsilon(A) = \rho^T q_\varepsilon = \varepsilon \sum_{k=0}^{\infty} (1-\varepsilon)^k \rho^T M^k \one_A.
\end{equation}

Regarding computational issues there is a decisive advantage of the $\varepsilon$-committors over the EMS time. As a consequence of the geometric averaging and the Neumann inversion formula, the operator 
\begin{equation}
  \varepsilon \sum_{k=0}^{\infty} (1-\varepsilon)^k M^k
\end{equation}
is invertible with inverse $ I - (1-\varepsilon) M$, and hence the $\varepsilon$-committor can be computed as the solution of the linear system~\eqref{eq:epscommittor} with equal complexity for every choice of $\varepsilon$. This differs greatly from the case of the finite-horizon EMS times where the corresponding operator is not invertible and with increasing time horizon the number of required matrix-vector multiplications increases as well. Future research will show whether these ideas can be combined with the results on the infinitesimal generator\cite{koltai2011stochastic, froyland2013estimating} in order to achieve a trajectory-free computation of the $\varepsilon$-committors.

\section{Convergence Results}\label{sec:converg}

In order to show that $\varepsilon$-committor and EMS time are natural generalizations of the notion of the basin of attraction, we will now prove convergence results for the geometric, respectively ergodic averages of operators related to them. We will see that under some assumptions the geometric, respectively ergodic averages of an operator $O$ converge to a projection onto the fixed space of $O$. Recall that if $O$ is a Koopman operator or an approximation thereof knowing its fixed space is equivalent to knowing, respectively approximating, the basin structure of the underlying dynamical system (see also Section~\ref{sec:ctsop} and Example~\ref{ex:emstulamfix}). These results imply that the quantities that we propose as notions of ``stochastic basins of attraction'', namely $\varepsilon$-committors and EMS times, converge back to the classical basins of attraction in the limiting cases.

Define for a bounded, linear operator $O$ on a general Hilbert space $H$ the ergodic mean $S_N[O] = \frac{1}{N} \sum_{k=0}^{N-1} O^k$ for $N\in \mathbb{N}$ and similarly the geometric mean $ C_\varepsilon (O) = \varepsilon \sum_{k=0}^\infty (1-\varepsilon)^k O^k$ for $\varepsilon \in (0,1]$. Then we have the following convergence results. 
\begin{mythm} (von Neumann mean ergodic theorem)

Let $O$ be a contraction, that is $\lVert Ov \rVert \leq \lVert v \rVert $ for all $v \in H$, then

\begin{equation}
 \lim_{N\to\infty} S_N[O]v = P_{\fix(O)}v = \lim_{\varepsilon\to 0} C_\varepsilon[O]v \quad \forall v \in H ,
\end{equation}
where $P_{\fix(O)}: H \to {\fix(O)}$ is the orthogonal projection onto the subspace ${\fix(O) = \{v \in H \mid Ov=v \}}$.
\label{thm:meanergo}
\end{mythm}
\begin{proof}
 The first equality is known as von Neumann's ergodic theorem\cite{krengel1985ergodic}. The second equality is proven analogously, see Appendix \ref{sec:ergocont}.
\end{proof}

\begin{myex} 
Any Markov operator and in particular the Perron-Frobenius operator $\pfo$ are by definition contractions on $L^1(X,\mu)$, compare Section~\ref{sec:markovop}. If $\mu$ is an invariant measure then $\pfo$, resp. the Koopman operator $\koo$, can be defined such that they are contractions on the Hilbert space $L^2(X,\mu)$, see \citet{lasota2013chaos} and \citet{eisner2015operator}.
\end{myex}

The same results hold for bounded, linear operators $O$ that can be written as the sum of a unitary and a part with spectral radius smaller than 1, see Appendix \ref{sec:ergodeco} for a proof.

\begin{mythm}
  Assume $H = \mathcal{U} \oplus \mathcal{V}$ is the direct sum of $O$-invariant closed subspaces $\mathcal{U}$ and $\mathcal{V}$, such that $O|_\mathcal{U}$ is unitary and $O|_\mathcal{V}$ has spectral radius smaller than 1, then
 
 \begin{equation}
  \lim_{N \to \infty} S_N[O]v = P_{\fix(O)}v = \lim_{\varepsilon \to 0} C_\varepsilon[O]v \quad \forall v \in H.
  \label{eq:operatorlimits}
 \end{equation}
\label{thm:operatorlimits}
\end{mythm}
Most relevant for our applications is the case when $O$ is a stochastic matrix representing a Markov chain and $H$ is $\mathbb{R}^n$. Then the last theorem yields the following result, which we prove in Appendix \ref{sec:ergomat}.
\begin{mythm}
 Let $M \in \mathbb{R}^{n \times n}$ be a stochastic matrix and let $Q \in \mathbb{R}^{n \times n}$ be invertible, such that $J=Q^{-1}MQ$ is the Jordan normal form of $M$, then
 \begin{equation}
  \lim_{N \to \infty} S_N[M] = \lim_{\varepsilon \to 0} C_\varepsilon[M] = P_{\fix(M)},
 \end{equation}
 where $P_{\fix(M)}$ is the projection onto ${\fix(M)}= \{v \in \mathbb{R}^n \mid v^T M = v^T\}$ given by
 \begin{equation}
  P_{\fix(M)} = Q^{-1} P_{\fix(J)} Q,
 \end{equation}
and $P_{\fix(J)}$ is an orthogonal projection.
\end{mythm}

\begin{myrem}
As for any closed subspace of a Hilbert space, there exists an orthogonal projection onto $\fix(M)$. 
However the projection $P_{\fix(M)}$ that we get from the theorem is in general not orthogonal.
This happens to be so, since stochasticity of a matrix is a property that only holds with respect to a certain basis of $\mathbb{R}^n$, namely the standard normal basis, where every basis vector corresponds to a certain state of the associated Markov chain. 
On the other hand the spectrum of a linear operator is independent of the basis and the theorem is mainly a consequence of the spectral properties of $M$.
We obtain $P_{\fix(M)}$ by switching to a suitable basis, such that $M$ has Jordan normal form, which allows us to apply the results of the previous theorem. \citet{meyer2000matrix} gives an explicit characterization of $P_{\fix(M)}$ in terms of sub-matrices of $M$. 
\end{myrem}

Above we saw that the right fixed points of the transition matrix associated to a Perron-Frobenius operator are expected to be almost constant on the basins, compare Section~\ref{sec:ctsop} and Example~\ref{ex:emstulamfix} (note that by duality these eigenvectors correspond to fixed points of the Koopman operator). For general Markov chains \citet{deuflhard2005robust} give some intuition on the structure of the right 1-eigenvectors.
In the ideal case of a Markov chain consisting of several uncoupled sub-chains, the right 1-eigenvectors will be constant on the irreducible components. If a Markov chain has several metastable states and transitions between these states are rare events, then it can be thought of as a small perturbation of such an ideal chain. \citet{deuflhard2005robust} show that for nearly uncoupled chains the perturbed 1-eigenvectors have eigenvalues close to 1 and are almost constant on the metastable components. They exploit this fact to approximate metastable states. However in the presence of long transients this constant level pattern is in general not preserved and more complex algorithms are required\cite{roblitz2013fuzzy, weber2015g}.

\begin{myrem} (Transient behaviour)

 We saw above that the asymptotic behaviour of EMS times and $\varepsilon$-committors is the same. In the presence of metastability we are able to derive an explicit expression for the difference between both quantities in the case that $\varepsilon = \frac{2}{N+1}$, see Appendix~\ref{sec:differenceestimates}. We find that for metastable states the difference is close to 0 for small $N$, then increases with growing $N$ and finally converges to 0 as expected. Overall the difference remains small. In the presence of a pronounced spectral gap, the combined error reaches a local minimum when $\varepsilon$ is chosen in between the different time scales.
\end{myrem}


\section{Generalized basin stability \label{sec:gbs}}

In this section we will define a generalized notion of basin stability that can be evaluated by sampling. In the preceding sections we studied discrete systems, and it is more technically involved to generalize the results rigorously to continuous time and state space. In contrast, the sampling procedure we consider here immediately generalizes to continuous time and state space. It also points towards a wider variety of ``$\varepsilon$-commitor-like functions'' that might be of interest for further study.

Basin stability is the probability that a deterministic system returns to a desirable attractor after a perturbation. Typically the perturbations are described by a probability density on phase space $\rho_\text{pert}(x)$. The basin stability $b$ is then simply given by the integral of the characteristic function of the basin $B$ with respect to $\rho_\text{pert}(x)$:

\begin{equation}
 b = \int \mathbbm{1}_B(x) \rho_\text{pert}(x) dx
\end{equation}

This integral can be evaluated using Monte-Carlo integration. Alternatively we can interpret the sampling directly as a Bernoulli experiment, drawing initial conditions and observing whether or not the system returns to the attractor. Crucially, we do not need to know the shape of the basin to estimate $b$. The relative accuracy of the unbiased estimator $\hat b (N_b)$ obtained by sampling $N_b$ trajectories is asymptotically small, and independent of system details. Specifically the standard error of the estimator is given by:

\begin{equation}
\label{eq:sampling-error}
 \sigma_{\hat b (N_b)} = \sqrt{\frac{\hat b (N_b) (1 - \hat b (N_b))}{N_b}} + O(N_b^{-1}) \;.
\end{equation}

In the case of general, not necessarily deterministic dynamics, the generalized basin stability of a set $A$ can be defined as 

\begin{equation}
\label{eq:gen_basin_stability}
 b_\text{gen} = \int q_\text{gen}(x) \rho_\text{pert}(x) dx \;,
\end{equation}
where $q_\text{gen}$ is the generalized membership function of the basin of $A$. In particular we can choose $q_\text{gen}$ to be the epsilon commitor $q_\varepsilon(x)$ or the expected mean soujourn time $s_T(x)$.

A Monte-Carlo estimation of this integral would be more expensive, as, for a stochastic system, $q_\varepsilon(x)$ or $s_T(x)$ can not be evaluated using only a single experiment. However, we can again design a Bernoulli experiment with expected probability $b_\text{gen}$. The experiment is as follows: Draw an initial condition from $\rho_\text{pert}$, run the system for a randomly chosen time $t$ and then check if it is in $A$ at that time.

To see this, note that the generalized membership functions themselves have the interpretation as the probability of a Bernoulli experiment. They correspond to the probability to run to $A$ in time $t$ when starting from some initial condition $x$ if we draw $t$ from an appropriate choice distribution $\rho_\text{run}$. For $q_\varepsilon(x)$ we take the exit time distribution $\epsilon e^{-\epsilon t}$ as the run time distribution $\rho_\text{run}$. This amounts simply to reinterpreting the exit from the system as run duration. For $s_T(x)$ we take the distribution of run times to be the equidistribution on the time interval $[0,T]$, so that the expectation value is equivalent to averaging in the time interval. These definitions in terms of probabilities naturally extend to continuous times:

\begin{align}
 q_\varepsilon(x) &= \int dt \, p(x(t) \in A|x(0) = x) \, p(t_\text{exit} = t)\nonumber\\
&= \int dt \, p(x(t) \in A|x(0) = x) \, \varepsilon e^{-\varepsilon t}\nonumber\\
 s_T(x) &= \int dt \, p(x(t) \in A|x(0) = x) \, \frac{1}{T}\mathbbm{1}_{[0,T]}(t)
\end{align}

Now the integral in \eqref{eq:gen_basin_stability} is simply given by drawing the initial condition from $\rho_\text{pert}$:
\begin{align}
 b_\text{gen} &= \int dt \int dx \, p(x(t) \in A|x(0) = x) \, \rho_\text{run}(t) \, \rho_\text{pert}(x)\;\nonumber\\
 &= p(x(T) \in A| p(x(0)) = \rho_\text{pert}, p(t) = \rho_\text{run})
\end{align}

This is the probability of a Bernoulli experiment, and thus can be studied by sampling again. We see immediately that this is true for a large class of such measures, namely, for all distributions of the evaluation time that are efficient to sample. Among these the two concepts developed in this paper are distinguished by taking the evaluation time to be either given by a constant stopping rate, or by a constant function.

This experiment will have the same variance of the estimator as the deterministic basin stability Eq. \eqref{eq:sampling-error}. Note further that after obtaining a sample of trajectories it is possible to evaluate the generalized basin stabilities for different sets on this sample, with each individual error is given by Eq. \eqref{eq:sampling-error}. However, the errors will be correlated, making it hard to do statistics on the various measures thus obtained.

\section{Examples\label{sec:applications}}
\subsection{Conceptual box models}

\begin{figure}
 \includegraphics[height=90pt]{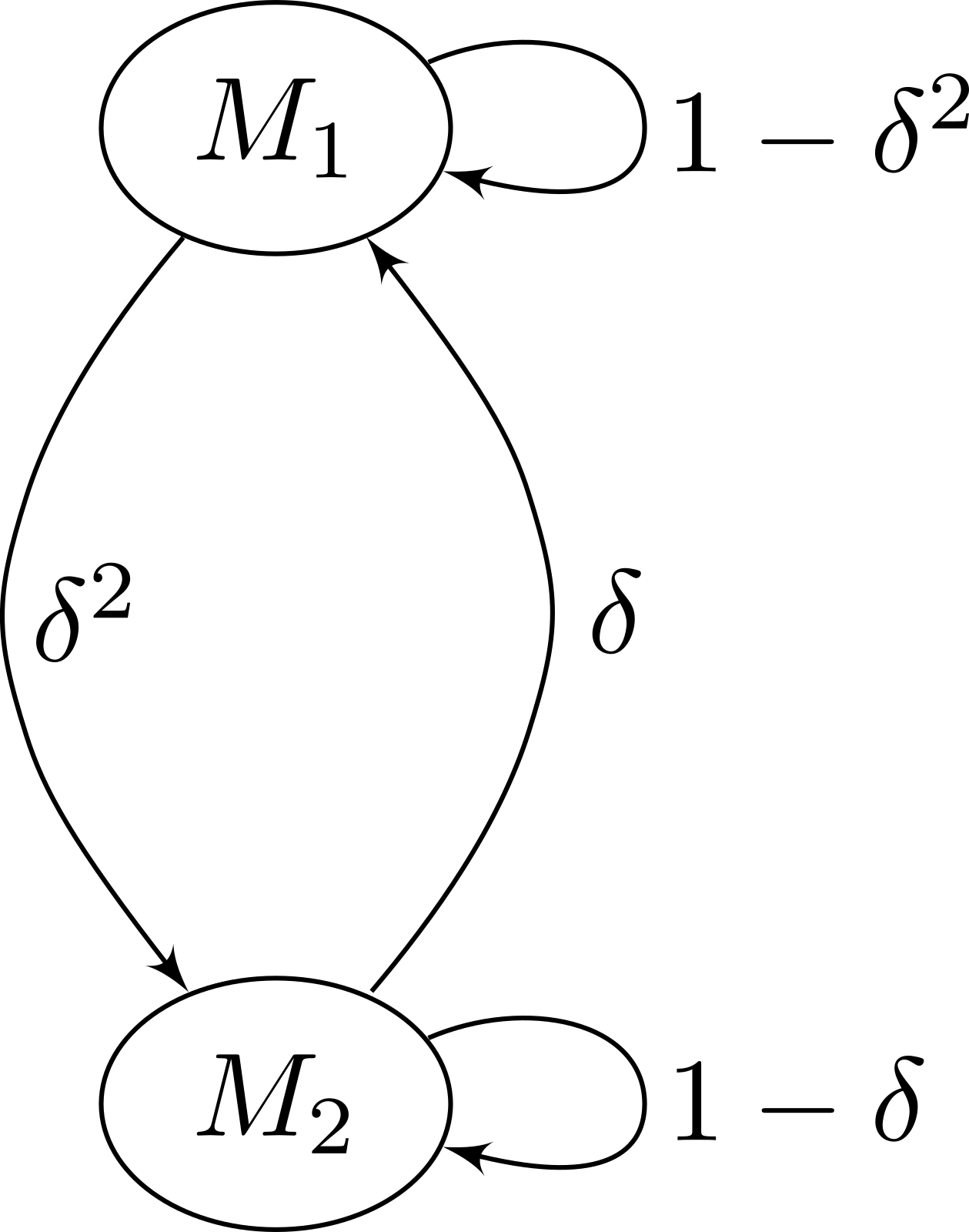}
 \hfill
 \includegraphics[height=90pt]{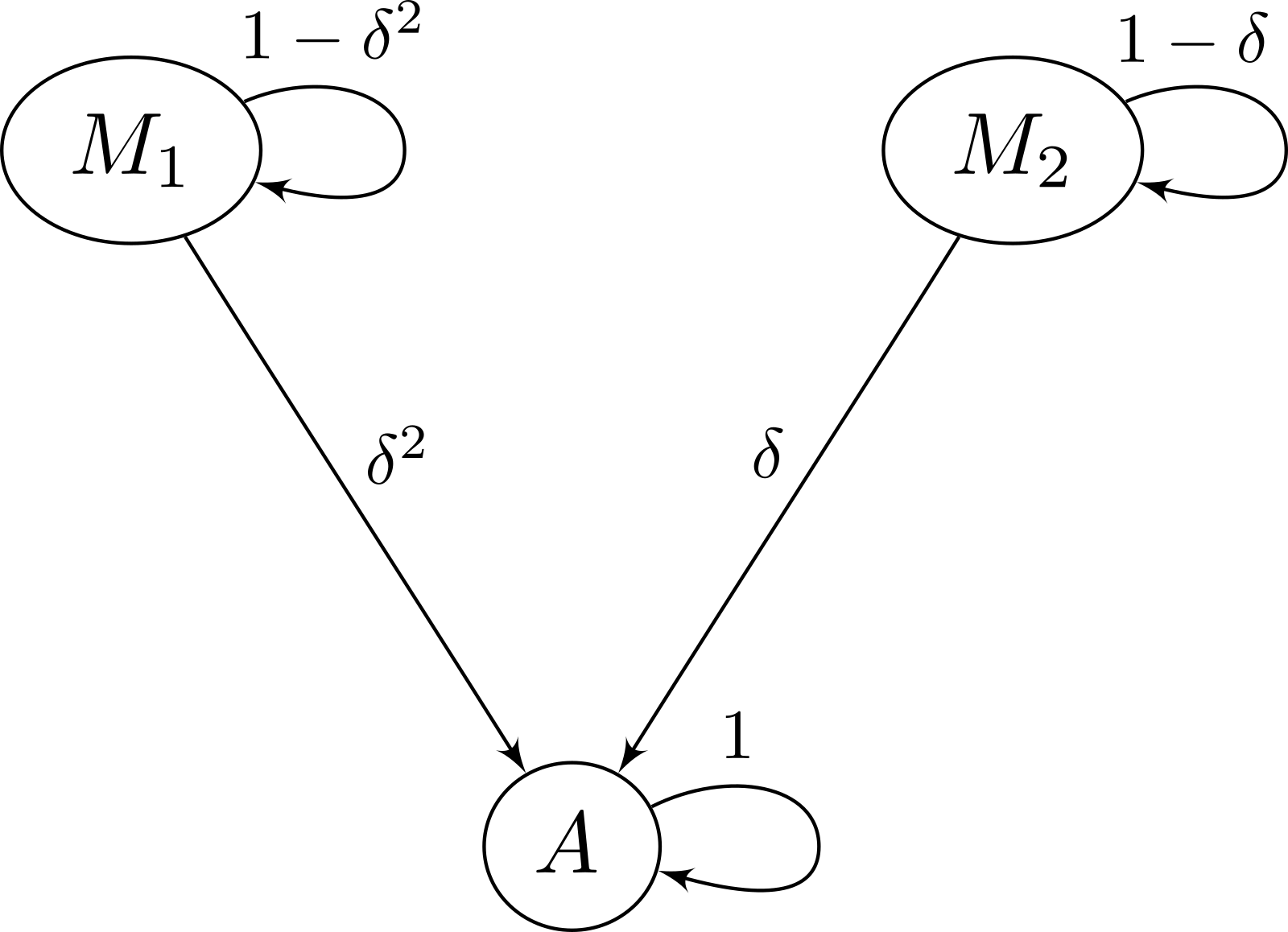}
 \caption{Conceptual models of metastability and long transients}
 \label{fig:conceptual}
\end{figure}
Figure~\ref{fig:conceptual} shows the state-transition diagrams of two very simple Markov chains. One of them consists of two almost-invariant states $M_1$ and $M_2$ and provides a conceptual model of metastability in an ergodic systems. The other one contains an attractor A and two transient states $M_1$ and $M_2$, where trajectories spend a long time before converging to A. The second model illustrates the concept of long transient states in dissipative systems.
The transition matrices are
\begin{equation}
 \begin{pmatrix}
  1 - \delta^2 & \delta^2    \\
  \delta       & 1 - \delta
 \end{pmatrix}\quad \mathrm{and} \quad 
 \begin{pmatrix}
  1 - \delta^2 & 0          & \delta^2 \\
  0            & 1 - \delta & \delta   \\
  0            & 0          & 1
 \end{pmatrix}.
\end{equation}

We compute $\varepsilon$-absorption stability with respect to normalized Lebesgue measure for these systems according to Eq.~\eqref{eq:epsabs} by solving
\begin{equation}
 b_\varepsilon(i) = \mathbbm{1}^T\frac{\varepsilon}{n}(I - (1-\varepsilon)M)^{-1}e_i,
\end{equation}
 where $M$ is the transition matrix and $e_i$ is the standard basis vector corresponding to state $i=M_1,M_2,A$ and $n=2,3$. The parameter $\delta$ controls the time-scales and is chosen to be 0.01 in the metastability model and 0.0001 in the long transients model. Figure~\ref{fig:conceptualepsabs} shows $\varepsilon$-absorption stability of the different states for varied $\varepsilon$.
 
 The limits of $\varepsilon(I - (1-\varepsilon)M)^{-1}$ for $\varepsilon$ to 0 are
 \begin{equation}
 \frac{1}{\delta^2 + \delta }
 \begin{pmatrix}
  \delta & \delta^2    \\
  \delta & \delta^2
 \end{pmatrix}\quad \mathrm{and} \quad 
 \begin{pmatrix}
  0            & 0          & 1 \\
  0            & 0          & 1 \\
  0            & 0          & 1
 \end{pmatrix},
\end{equation}
 and it follows that the 0-absorption stability $b_0$ is
 \begin{equation}
  b_0 = \lim_{\varepsilon \to 0} b_\varepsilon =   \frac{1}{\delta^2 + \delta }
  \begin{pmatrix}
  \delta           \\
  \delta^2         \\
 \end{pmatrix} \quad \mathrm{and} \quad 
  \begin{pmatrix}
  0            \\
  0            \\
  1
 \end{pmatrix},
 \end{equation}
 which are just the invariant distributions. This shows that for the long transient model we recover the usual basin stability value in the $\varepsilon$ to $0$  limit. For the metastability model basin stability is not well-defined since trajectories never converge to an attractor. In this case $\varepsilon$-absorption stability converges to the invariant distribution.

Figure~\ref{fig:conceptualepsabs}b shows that for finite $\varepsilon$ the attraction of a region on that time scale is accurately captured. With $\varepsilon$ between $\delta$ and $\delta^2$, which are indicated by the vertical lines, the committor sees that region $A$ is attracting $M_2$, but not $M_1$. Conversely, on these timescales, $M_1$ is stable.
 
\begin{figure}
 \includegraphics[width=0.49\columnwidth]{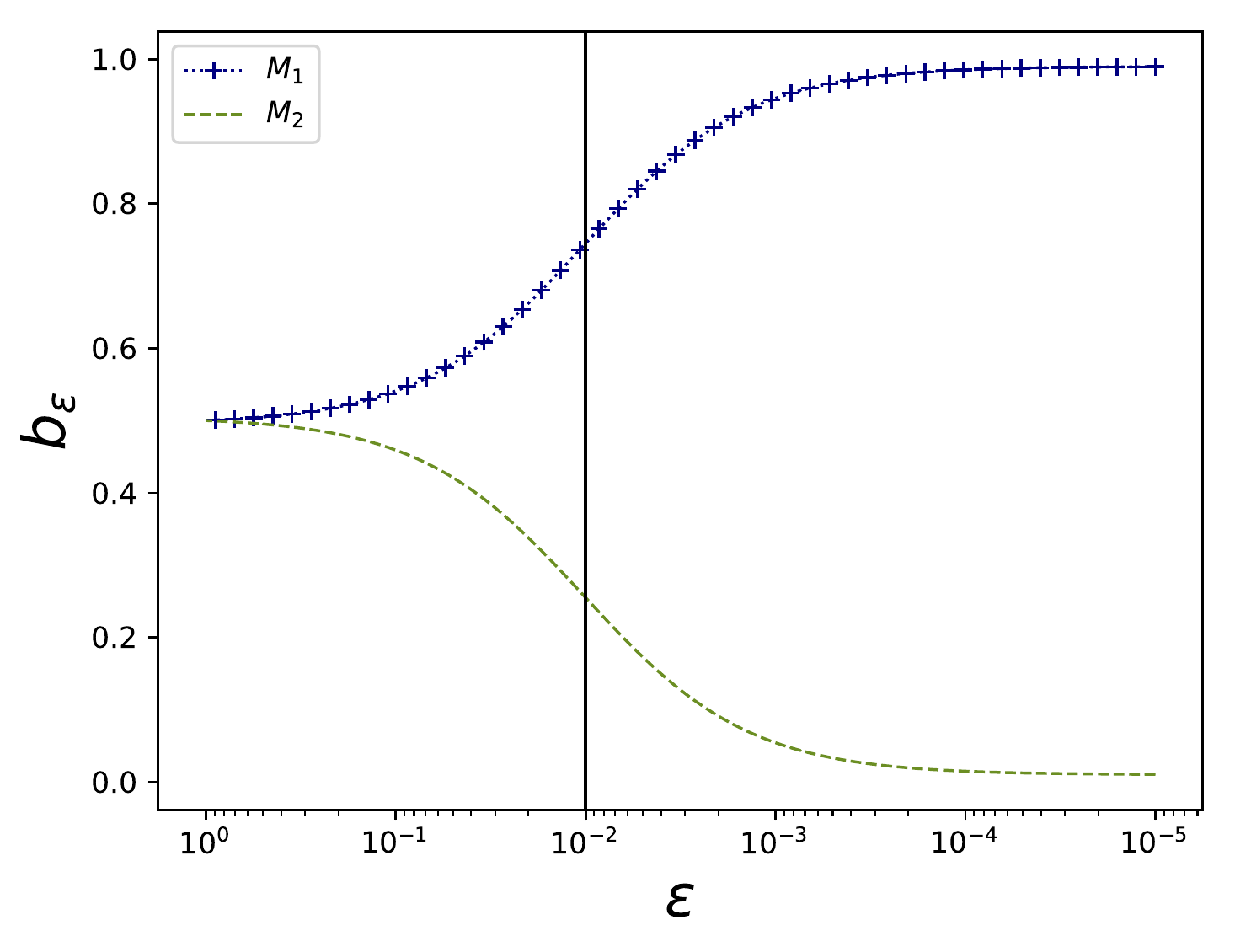}
 \hfill
 \includegraphics[width=0.49\columnwidth]{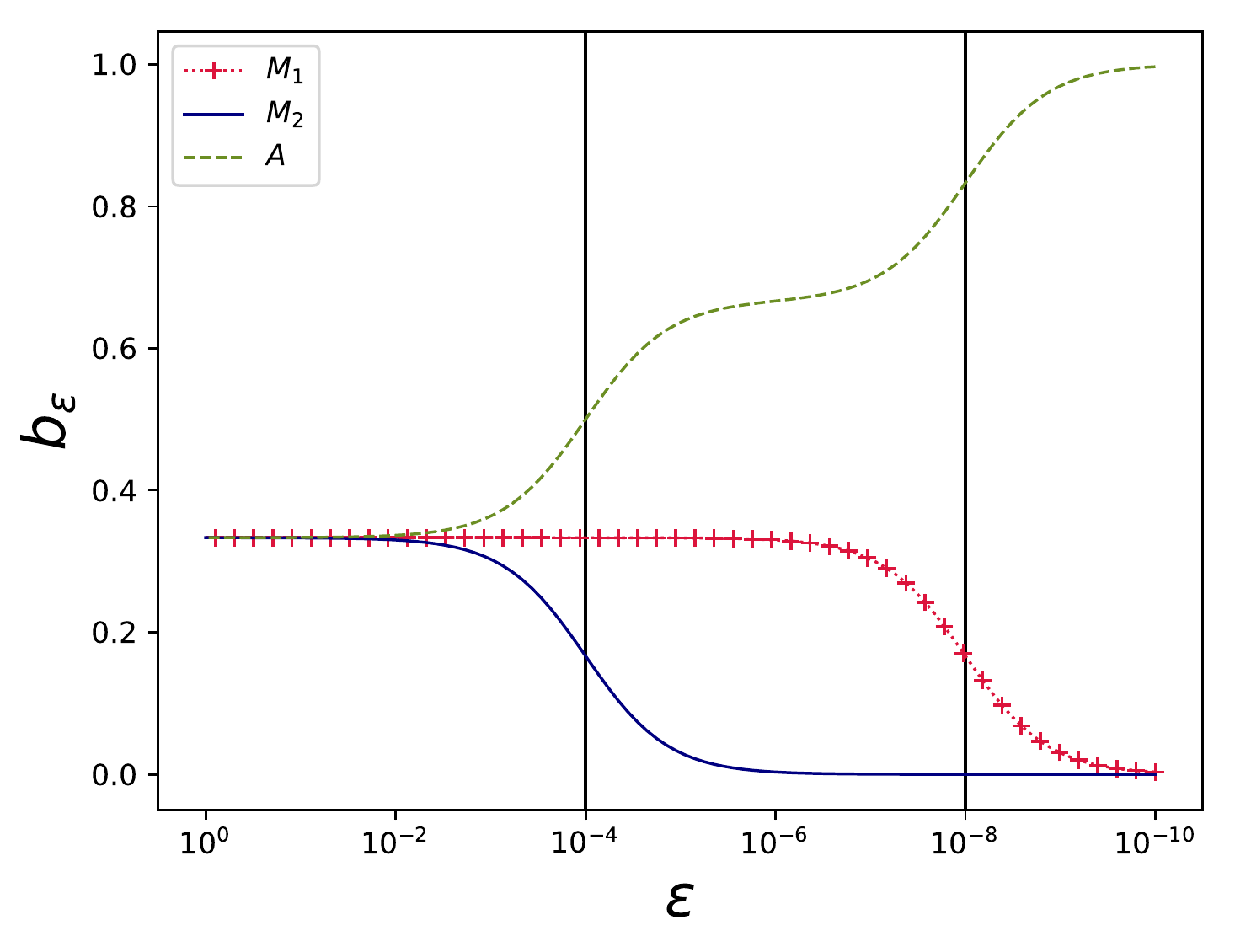}
 \caption{$\varepsilon$-absorption stability for the model with almost-invariant states (left) and long transient states (right). The initially identical values converge asymptotically to the invariant distribution with decreasing $\varepsilon$. For the metastability model the invariant distribution on $M_1$ is close to 1, but strictly smaller. In the long transient model $b_\varepsilon(M_1)$ stays almost constant over a large interval of $\varepsilon$, since the leak rate $\delta^2$ is very small. The horizotal lines indicate the value of $\delta$, respectively $\delta^2$.}
\label{fig:conceptualepsabs}
\end{figure}

\subsection{Damped driven pendulum}
The following sytem of equations describes the dynamics of a damped driven pendulum\cite{menck2013basin} and is used in classical power grid models to model a single generator~\cite{menck2014dead}.
\begin{align}
 \dot{\phi }  &= \omega \nonumber \\
 \dot{\omega} &= -\alpha\omega + P - K\sin\phi,
\end{align}

The parameter values are $\alpha=0.1, K=1$ and $P=0.5$. The system has a stable fixed point at $(\phi,\omega) = (\sin^{-1} \frac{P}{K},0)$ and a stable limit cycle at approximately $(\phi,\omega) \approx (\phi, 5)$, compare Figure~\ref{fig:ddp} for a plot of the phase space.
\begin{figure} 
\centering
 \includegraphics[width=0.6\columnwidth]{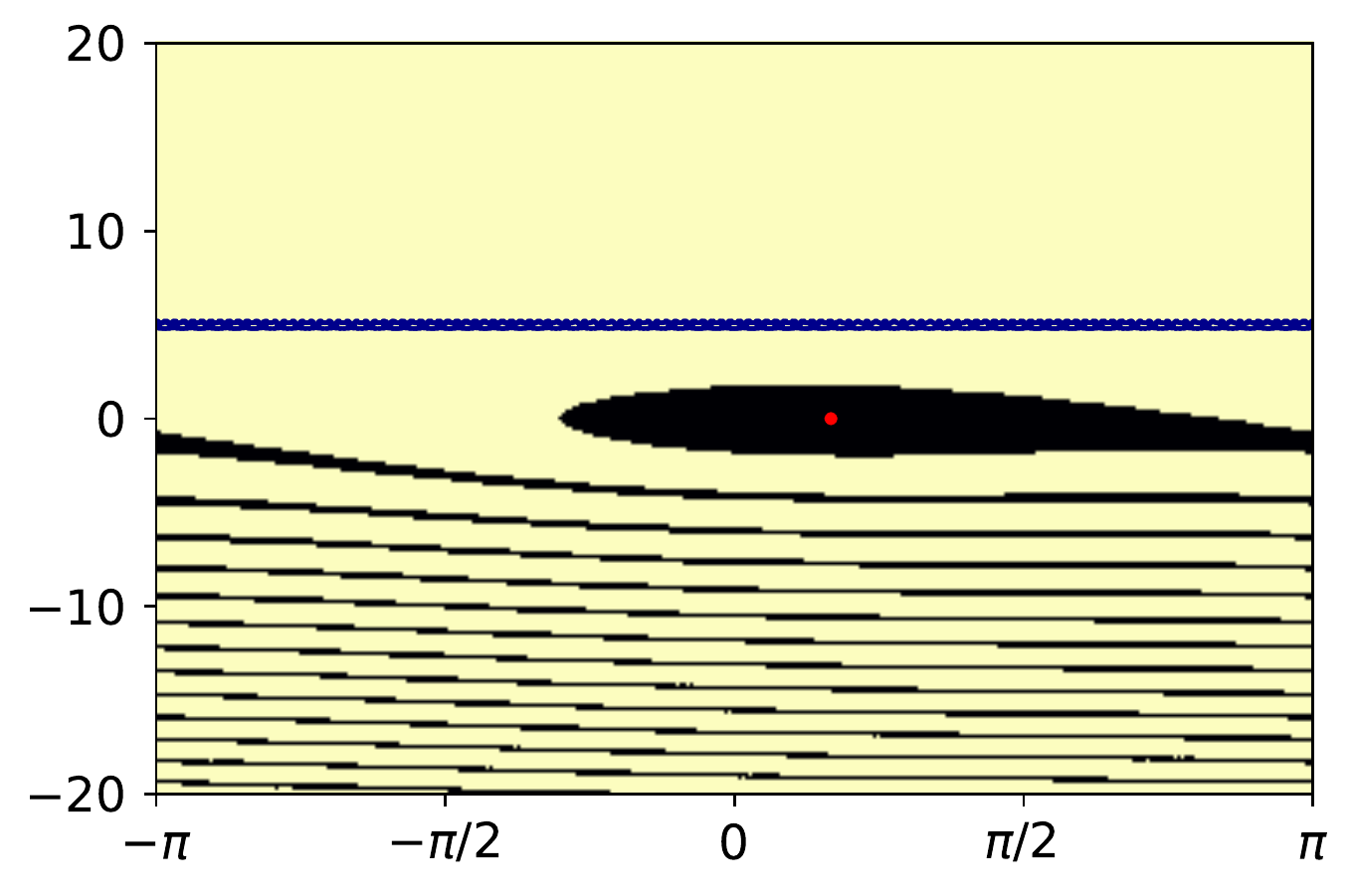}
 \caption{Basins of attraction with corresponding fixed point (red) and limit cycle (blue) of the damped driven pendulum at parameter values $\alpha=0.1, K=1$ and $T=0.5$.}\label{fig:ddp}
\end{figure}

\begin{figure} 
\includegraphics[width=0.49\columnwidth]{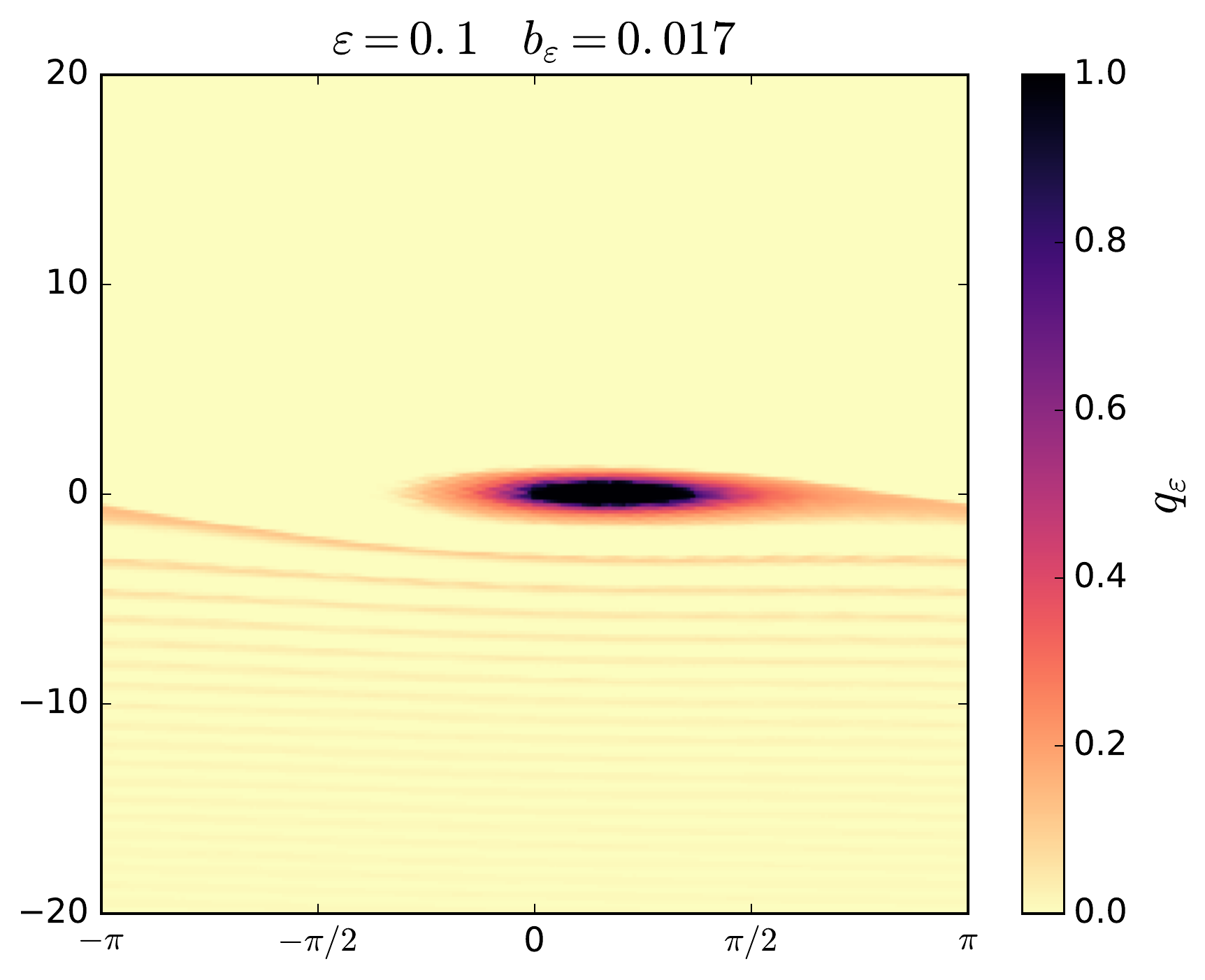}\hfill
\includegraphics[width=0.49\columnwidth]{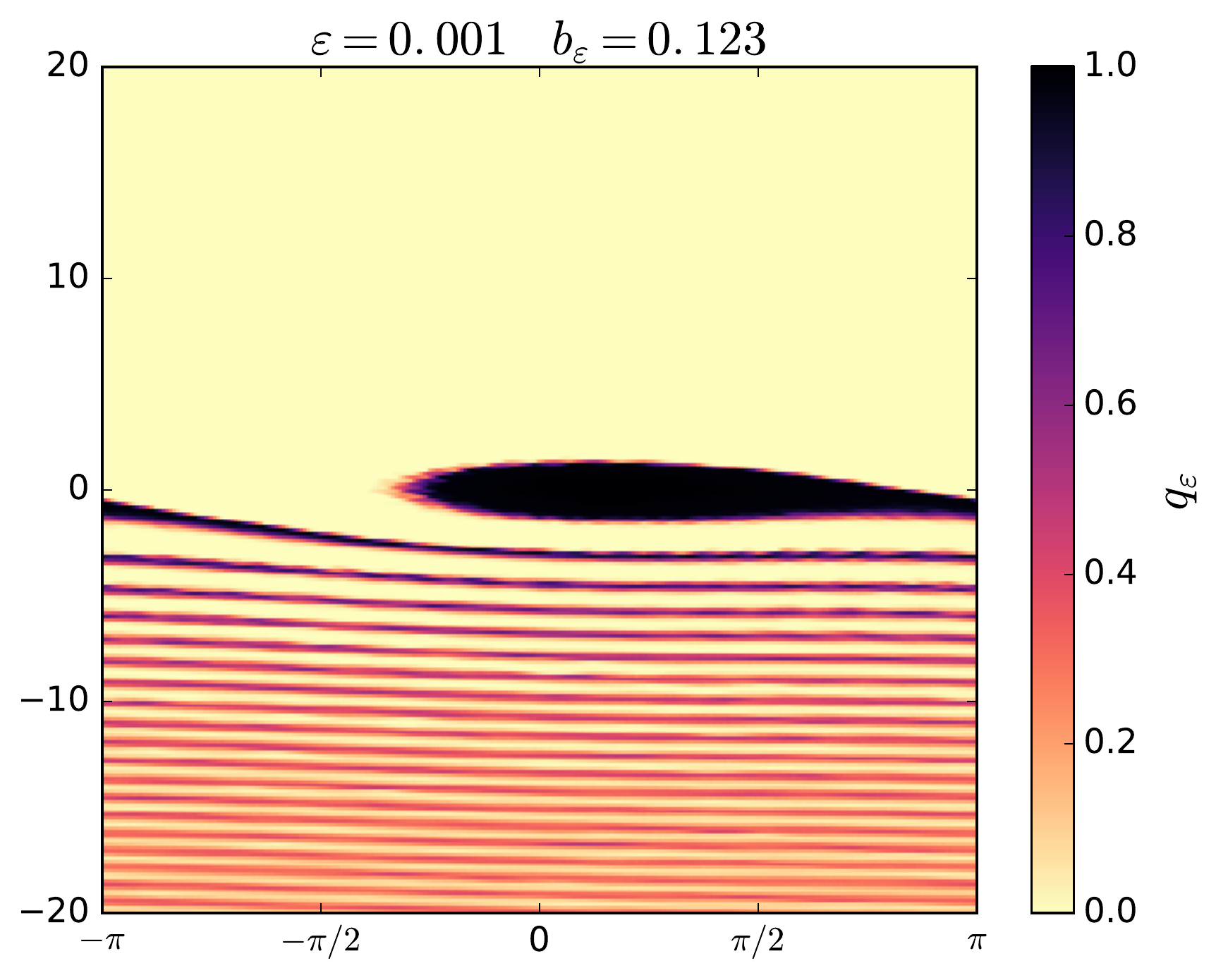}
\caption{$\varepsilon$-committors of the metastable set around the fixed point. The absorption rates are $\varepsilon=0.1$ and $\varepsilon=0.01$ with corresponding $\varepsilon$-absorption stability values of $b_\varepsilon=0.017$ and $b_\varepsilon=0.123$. For $\varepsilon=0.01$ the whole basin of the fixed point is detected and $b_\varepsilon$ is close to the basin stability value of the original system. If $\varepsilon$ is decreased further the basin stays qualitatively unchanged, while the basin stability value is even better approximated.}\label{fig:ddpcom}
\end{figure}

In order to compute $\varepsilon$-absorption stability, first we have to transform the ordinary differential equation (ODE) into a discrete dynamical system. For a fixed timestep $\tau$ the flowmap $\varphi(x_0) = \varphi(\tau, x_0) = x(\tau)$ gives the value $x(\tau)$ at time $\tau$ of a solution $x(t)$ of the ODE with initial condition $x_0$. Then $\varphi: X \to X$ defines a discrete dynamical system and we can construct the Perron-Frobenius operator and its Ulam approximation according to Section~\ref{sec:ulam}ff.

\sloppy
For Ulam's method we use $256\times256$ regular square boxes on the state space ${[-20,20]\times[-\pi,\pi]}$, such that the resulting transition matrix has dimension ${65536 \times 65536}$. In every box 1000 initial conditions are initiated uniformly on random and numerically integrated for the time-step $\tau=1$ in order to obtain the transition probabilities between boxes.

The discretization by Ulam's method introduces discretization diffusion in the system and thereby destroys the stability of the attractors, in particular of the stable fixed point, since trajectories in its basin spiral only slowly towards it and therefore it is possible that they enter a box centered outside the fixed points' basin. However, a metastable set remains in the vicinity of the fixed point. By analyzing this set we can determine the basin of attraction of the original fixed point. If we choose less boxes for our discretization method, the resulting discretization noise increases and metastability of the set around the fixed point decreases until its relation to the deterministic behaviour is lost. 

Figure~\ref{fig:ddpcom} shows the committor functions of the metastable set around the fixed point. As expected the basin of $\varepsilon$-absorption converges to the basin of attraction shown in Figure~\ref{fig:ddp} when the expected time horizon is increased. The $\varepsilon$-absorption stability value of $b_\varepsilon \approx 0.1262$ for $\varepsilon = 10^{-8}$ is in very good accordance with the classical basin stability value obtained by Monte Carlo integration as $0.1267 \pm 0.0002$. Note that the required number of function evaluations in order to compute basin stability up to this precision by the Monte Carlo approach is considerably higher than the number of function evalutions required to construct the transition matrix. When $\varepsilon$ is further decreased the values of the $\varepsilon$-committor are expected to slowly decrease due to discretization diffusion. At an resolution of $256 \times 256$ boxes this effect is not observed since it is below numerical precision, however at a resolution of $128 \times 128$ boxes it is clearly visible and for resolutions below $64 \times 64$ boxes discretization diffusion gets too strong to draw any reliable conclusions on the systems dynamics.

Obviously the number of boxes is the main factor for determining the computational cost of Ulam's method and hence it is desirable to use as few boxes as possible. For some systems adaptive partitions may greatly reduce computational effort by using fewer partition elements\cite{froyland2001extracting, dellnitz1998adaptive}.

\subsection{A chain of oscillators}

In order to illustrate the sampling approach for high dimensional systems we study a chain of 16 coupled damped driven pendula subject to additive noise, and perturb them around the synchronous state. Perturbations are from the range $\pm 5$Hz and $\pm \pi$. We study the system with $P = 1$ and $K = 8$ for various levels of additive noise acting on the frequencies. The results for various choices of time horizon/absorption probability $T = \frac{2}{\varepsilon}$ and noise strength $\sigma$ are shown in Figure~\ref{fig:gen-bs}. The region whose basin of attraction is studied is that of all frequencies smaller than $0.5$Hz. This is qualitatively the type of constraint on the behaviour of a system that one is concerned about in the context of power grid modelling.

Note that, as can be seen from the single damped driven pendulum, the region around the attractor is only metastable if noise is added to the system. Therefore this is an example of generalized basin stability for a metastable state.

Looking at low noise, the probability to end in the region studied first increases with $T$. This shows the time scale on which the perturbations studied return to the metastable region. As the fixed point is the only attractor in the region, the no-noise stochastic basin stability converges to the basin stability of the attractor as $T$ increases. With some noise added the stochastic basin stability remains close to the deterministic one, until we see the noise reach a strength where the metastability of the region studied collapses. This illustrates that our stochastic basin stabilities are a natural generalization of basin stability.

\begin{figure}
\includegraphics[width=\columnwidth]{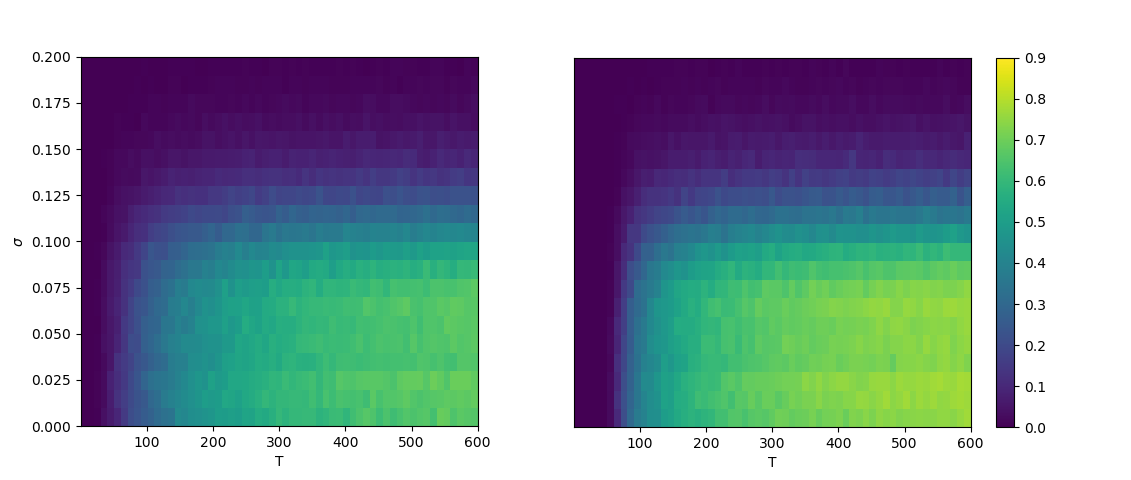}
\caption{Stochastic basin stability based on the mean-soujourn time (left) and the $\varepsilon$-committor (right) as a function of time horizon $T = \frac{2}{\varepsilon}$ and noise strength $\sigma$. Region $A$ is the part of phase space satisying $|\omega_i| < 0.5Hz$, close to the synchronous state. Bottom right corner converges to deterministic basin stability.}
\label{fig:gen-bs}
\end{figure}

\subsection{Anderies' model of global carbon dynamics \label{sec:anderies}}

\citet{anderies2013topology} introduce a non-linear conceptual model of global carbon dynamics that exhibits long transient trajectories when started in a particular region of phase space. The model equations for marine $c_m$, terrestrial $c_t$ and atmospheric $c_a$ carbon are

\begin{alignat}{2}
 & \dot{c}_m &&= \alpha_m (c_a - \beta c_m) \nonumber \\
 & \dot{c}_t &&= \mathrm{NEP}(c_a,c_t) - \alpha c_t \nonumber \\
 & c_a       &&= 1 -c_m -c_t, \nonumber
\end{alignat}
where $\alpha_m=0.05$, $\alpha=0.1$ and $\beta=1$ and $\mathrm{NEP}$ denotes a complex, non-linear relation between $c_a$ and $c_t$, which is explained in detail in~\citet{anderies2013topology}. Due to the third equation the total amount of carbon stays constant and we can consider the system on the restricted phase space $X=\{ (c_m,c_t) \in [0,1]^2 \mid c_m + c_t \leq 1 \}$. For the chosen parameters the system has a single, globally attractive fixed point and hence basin stability equals 1 by definition. Trajectories starting with low marine and terrestrial carbon stocks, i.e. $c_m + c_t \leq 0.4$ pass through a set where $c_t \approx 0$ before converging to the stable state. 

The so-called dead zone is defined as $D:=\{(c_m,c_t) \in X \mid c_t < 0.1\}$ and corresponds to a state of low terrestrial carbon stocks, i.e. when pratically all land-based vegetation and thus the basis for human life has vanished. It contains a long transient region where some trajectories spend a large amount of time before they converge to the attractor. Since the probability that the process is in $D$ decreases monotonically for large time-horizons, we can obtain lower bounds for the expected time the process spends in $D$ during the first $\varepsilon^{-1}$ steps by applying Example~\ref{ex:longtransientsiii} and computing the normalized $\varepsilon$-committor $\frac{1}{\varepsilon}q_\varepsilon$. Assuming a society is able to survive a state of low-terrestrial carbon given that vegetation recovers fast enough, the $\varepsilon$-committors may be used to assess which trajectories are ``survivable'', thus complementing the notion of ``survivability'' for dynamical systems recently introduced by~\citet{hellmann2016survivability}.

In order to compute the $\varepsilon$-committors we discretize the square $[0,1]^2$ into $128\times128$ uniform square boxes, discard all boxes that have empty intersection with $X$ and compute the transition matrix according to Section~\ref{sec:ulam}ff. The resulting partition has 8256 elements, where the boxes on the diagonal are triangles with half the weight of a square box.

Figure~\ref{fig:anderies} shows the classical committor function of the dead zone that we introduced in Section~\ref{sec:classcomm} along $\frac{1}{\varepsilon} q_\varepsilon$ for different values of $\varepsilon$. We chose to show $\frac{1}{\varepsilon} q_\varepsilon$ over $q_\varepsilon$ since for transient sets the latter simply tends to zero, while the former stabilizes at the expected time the process with absorption spends in the set $D$, cf. Section~\ref{sec:furthercommis}. Given that $\varepsilon^{-1}$ is large enough, this provides a lower bound for the expected time the original process (without absorption) spends in $D$, and even more, by the reasoning that we applied in Section~\ref{sec:differenceestimates}, it converges to the same value for~$\varepsilon$~to~0.

\begin{figure}
\includegraphics[width=0.49\columnwidth]{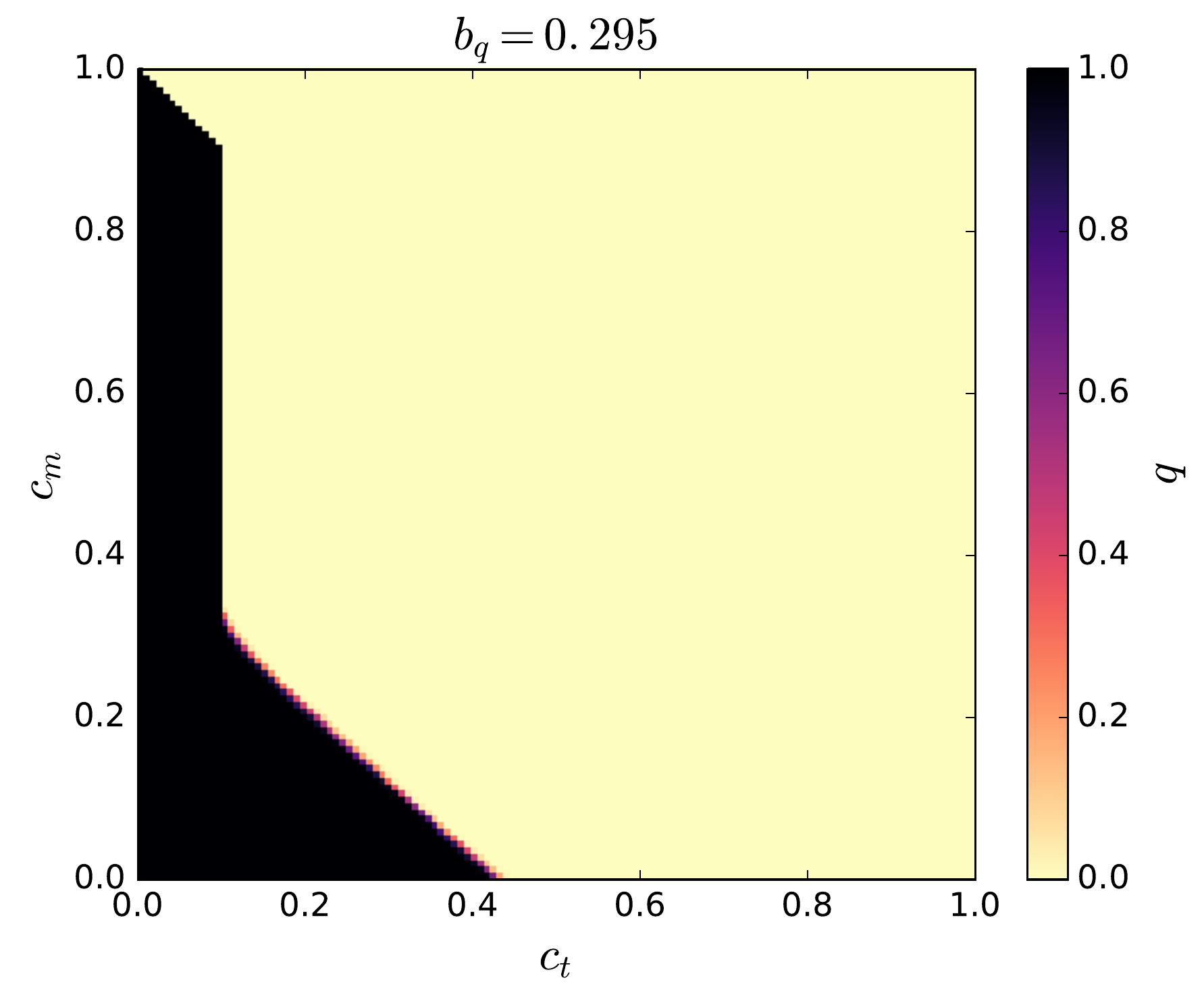}
\includegraphics[width=0.49\columnwidth]{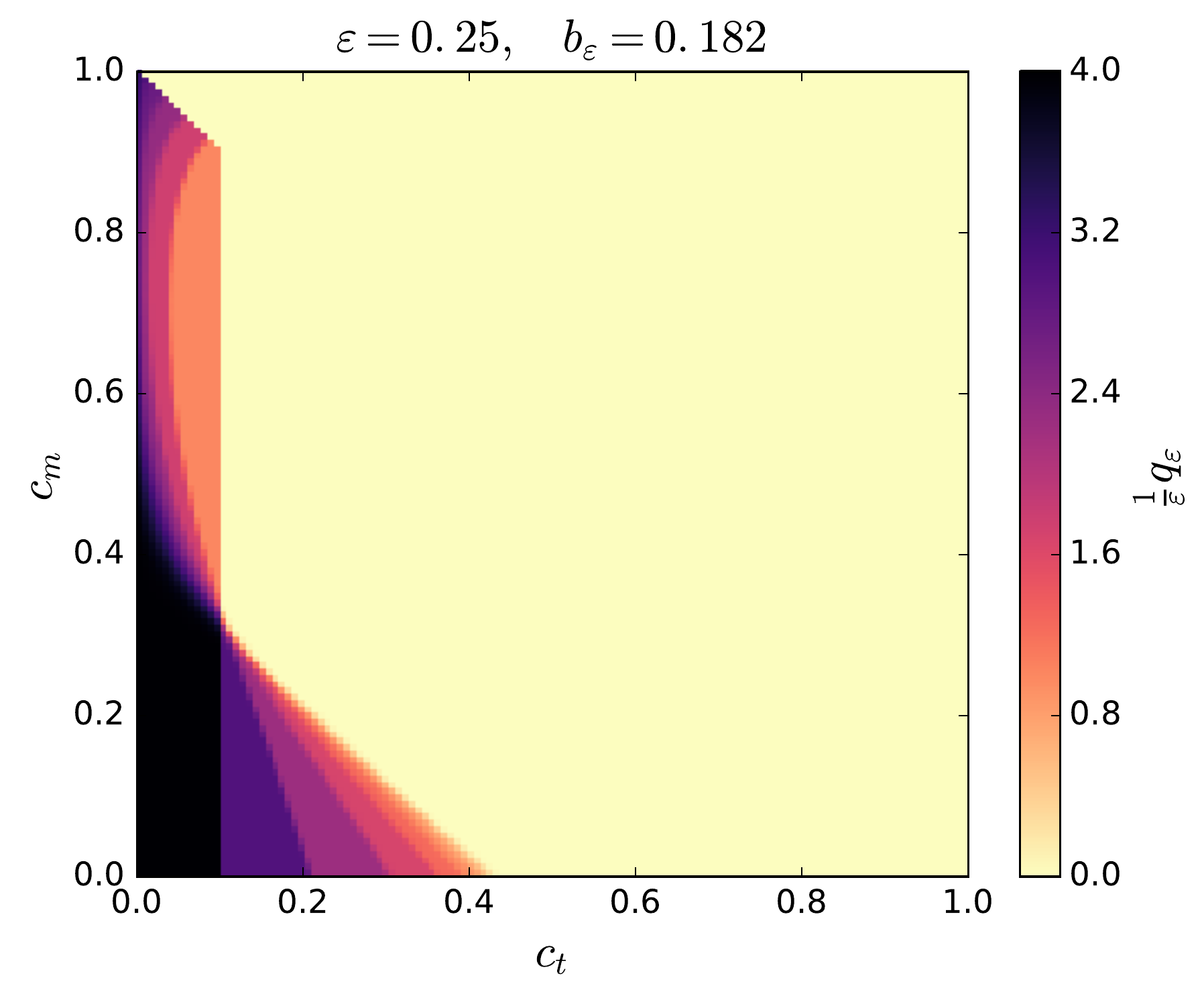} \\
\includegraphics[width=0.49\columnwidth]{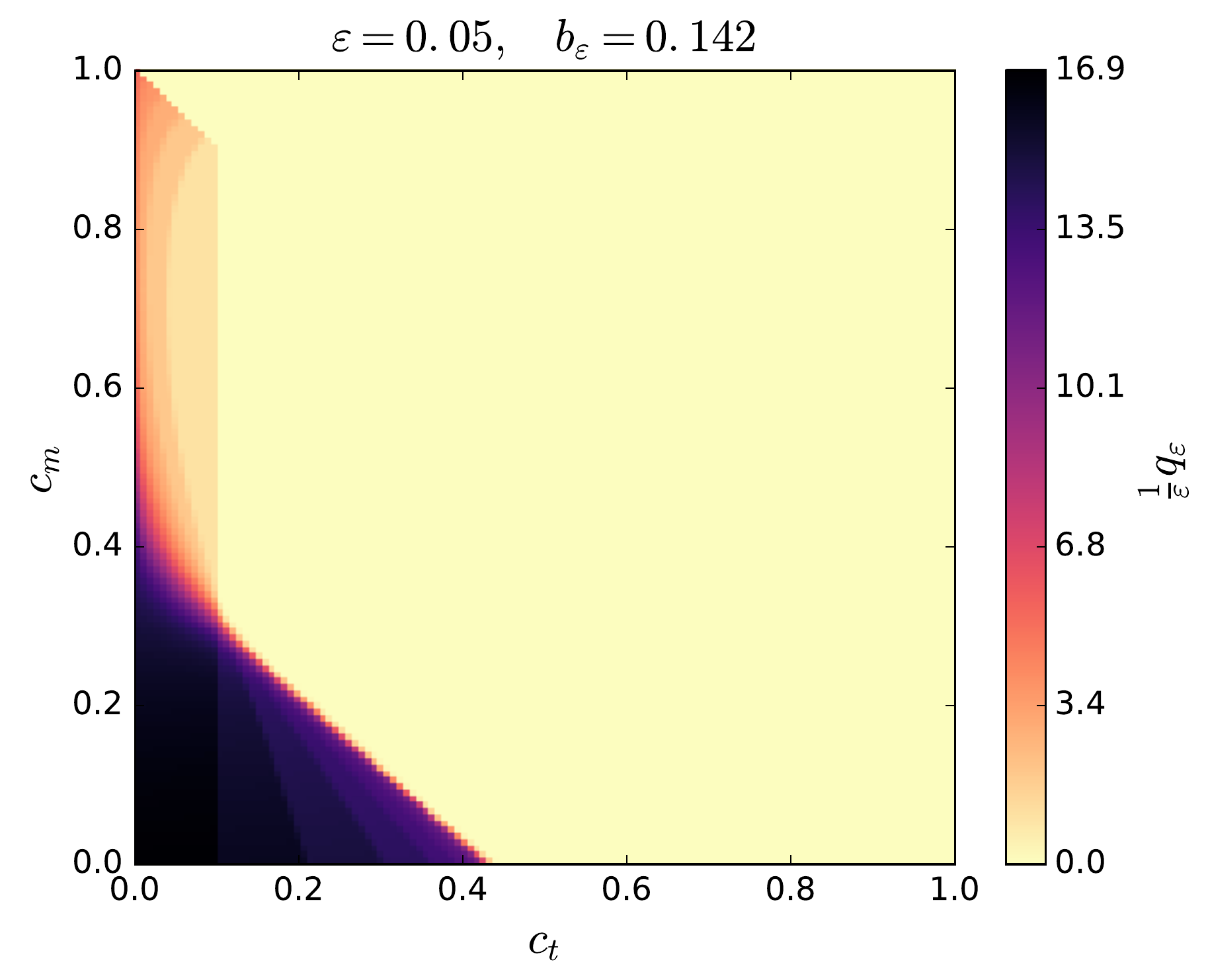}
\includegraphics[width=0.49\columnwidth]{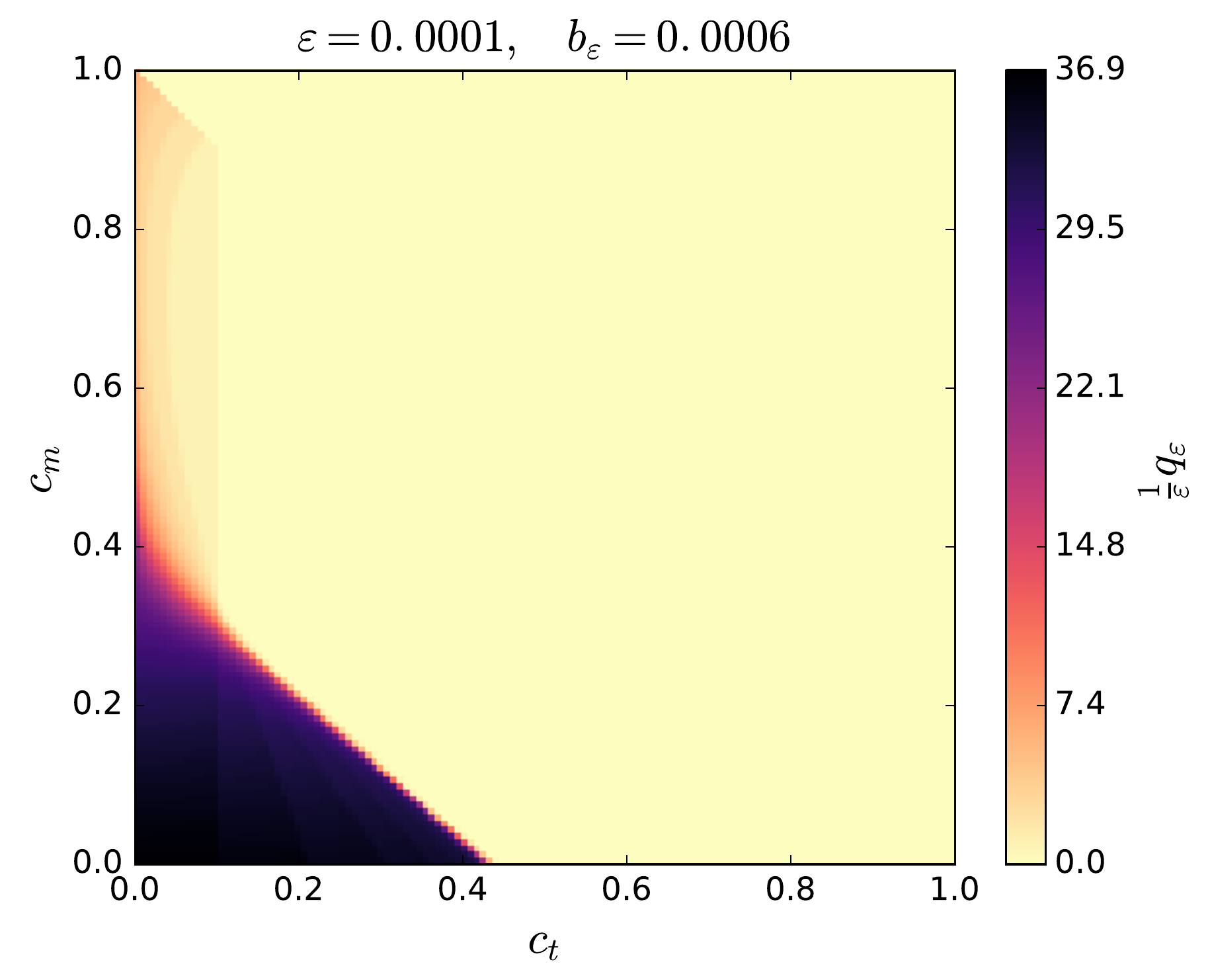}
\caption{The upper left plot shows the classical committor function $q$ with respect to the dead zone $D$ and the asymptotic fixed point, cf.~Eq.~\eqref{eq:committorAB}. The fraction of initial states that eventually hit the dead zone is $0.295$. The other plots show $\varepsilon$-absorption stability for $\varepsilon=0.25, 0.05, 0.0001$ and the expected time the process spends in $d$ before absorption, as described by the normalized $\varepsilon$-committors $\frac{1}{\varepsilon} q_\varepsilon$.   While $b_\varepsilon$ tends to zero for $\varepsilon$ to zero the maximum value of $\frac{1}{\varepsilon} q_\varepsilon $ converges to $36.9$.}
\label{fig:anderies}
\end{figure}

\section{Conclusions}

We introduced transfer operator methods for dynamical system and established connections from basins of attraction to related concepts in dynamical systems, functional analysis and Markov chain theory. On this basis we developed the novel concept of $\varepsilon$-committors and studied their general properties as well as asymptotic behaviour. We saw that the $\varepsilon$-committors generalize basins of attraction for systems with long transients or metastable states. They can be applied to stochastic and deterministic systems likewise. Their connection to mean sojourn times was investigated in detail for metastable states. $\varepsilon$-committors proved especially useful in applications with an undesirable region in phase space, since they allow to compute the time the process is expected to spend in this region. We highlight again that only short trajectories are needed for computing $\varepsilon$-committors and that they give access to transient properties of the system at every timescale with equal computational effort. 

Importantly we showed that the basin stability for these stochastic basins of attraction can be estimated at comparable cost to deterministic systems.

Compared to the work of \citet{serdukova2016stochastic} our definition is entirely intrinsic and does not presupose knowing the basin of attraction. This allows for a straightforward estimator for the generalized basin stability, whereas it is not known whether such an estimator exists for the definition of \citet{serdukova2016stochastic} (see \citet{schultz2017bounding} though for an estimator for a related quantity). We define stochastic basins more generally for measureable sets in arbitrary stochastic systems given that their evolution is described by a Markov operator. The trade off is that our stochastic basin requires a choice of region and will in general depend on this choice. We leave working out the precise relationship between these two notions of stochastic basin to future work.

While the probabilistic formulation of the $\varepsilon$-committors generalize immediately to systems on continuous state spaces and for continuous-time dynamical systems, it would be interesting to also develop the appropriate PDE formulations for them, as well as for the fuzzy committors. Another interesting question is if $\varepsilon$-commitors can be used to define metastable sets via a minimization problem. 

We have shown that the transfer operator approach provides a conceptual framework for stochastic basin stability, but we also hope that it might eventually be a theoretical foundation for the development of more efficient algorithms to evaluate stochastic and deterministic basin stability for systems with long transients. To estimate basin stability requires the integration of trajectories until they have converged to the attractor (up to numerical precision). For systems with long transient states trajectories might be very expensive to compute, while transfer operators capture all timescales without requiring long trajectories. At the present moment this approach works efficiently in low-dimensional state spaces, with the trade-off being that numerical diffusion blurs the basin \cite{koltai2011stochastic, froyland2013estimating}.  It cannot be applied to high-dimensional systems since the computational cost of Ulam's method increases exponentially with the dimension of state space.

We hope that eventually transfer operator methods will facilitate the development of efficient algorithms for estimating basin stability and related measures like $\varepsilon$-absorption stability in high-dimensional systems, possibly in conjunction with techniques from randomized linear algebra~\cite{mahoney2016lecture}. 

\section*{Software}
The simulations were performed using Julia and Python, using the SciPy package \cite{scipy}. The high dimensional example was implemented using the DifferentialEquations.jl library\cite{rackauckas2017differentialequations, rackauckas2017adaptive} using the algorithms of \citet{rossler2010runge}.

\section*{Acknowledgments}
The authors would like to especially thank P\'eter Koltai for many extensive helpful discussions on the use of committors in the context of basins, in which the generalized committors were defined.

We would also like to thank Jobst Heitzig and Paul Schultz for extensive comments on the final draft of this manuscript and Chris Rackauckas with help in implementing the example C. Parts of this work were funded by BMBF (CoNDyNet Grant No. 03SF0472A, CoSy-CC2 Grant No. 01LN1306A), Volkswagen Foundation (Grant No.  88462) and the Deutsche Forschungsgemeinschaft (Grant No. KU 837/39-1 / RA 516/13-1).
\bibliographystyle{unsrtnat}
\bibliography{bibdat.bib}

\appendix

\section{Convergence results \label{sec:convergenceresults}}

In this section we will prove convergence results for the geometric, respectively ergodic averages of operators related to $\varepsilon$-committor and EMS time.  We develop the theory in a general functional analytic setting since then the structure of the proofs is clearer. At the same time the results are more profound and might serve as a stepping stone for extending our concepts to transfer operators acting on infinite-dimensional spaces. Sections~\ref{sec:ergocont}~to~\ref{sec:ergodeco} establish ergodic theorems for special classes of Hilbert space operators, Section~\ref{sec:ergomat} focuses on the important application case of stochastic matrices. Most importantly we will see that under some assumptions the geometric, respectively ergodic averages of an operator $O$ converge to a projection onto the fixed space of $O$. Recall that if $O$ is a Koopman operator or an approximation thereof knowing its fixed space is equivalent to knowing, respectively approximating, the basin structure of the underlying dynamical system (see also Section~\ref{sec:ctsop} and Example~\ref{ex:emstulamfix}). These results imply that the quantities that we propose as notions of ``stochastic basins of attraction'', namely $\varepsilon$-committors and EMS times, converge back to the classical basins of attraction in the limiting cases.

\subsection{Ergodic theorems for contractions on a Hilbert space \label{sec:ergocont}}

This paragraph follows the approach taken by~\citet{krengel1985ergodic}. 
Let H be a Hilbert space, and denote the scalar product of $u,v \in H$ as $\langle u,v \rangle$. $\mathcal{B}[H]$ is the set of bounded, linear operators $O:H\to H$. Denote by $O^*$ the dual of $O \in \mathcal{B}[H]$, such that $\langle Ou,v \rangle = \langle u, O^*v\rangle \quad \forall u,v$. 

The norm of an operator $O \in \mathcal{B}[H]$ is given by
\begin{equation}
 \lVert O \rVert = \sup_{\lVert v \rVert \leq 1} \lVert Ov \rVert.
\end{equation}

\begin{mylem}\label{lem:adjnorm}
 If $O$ is a bounded, linear opertor on $H$ and $O^*$ its dual, then
 \begin{equation}
  \lVert O \rVert = \lVert O^* \rVert
 \end{equation}
\end{mylem}

\begin{myex}
 If $M$ is a real matrix then its dual operator is the transposed matrix $M^T$.
\end{myex}

$O$ is called \emph{contraction}, if $\lVert O \rVert \leq 1$. A bounded, linear operator $U$ is called unitary if $U$ is surjective and preserves the scalar product, i.e.
\begin{equation}
 \langle Uu, Uv\rangle = \langle u,v \rangle \quad \forall u,v.
 \end{equation}
An unitary operator is a contraction and its specturm lies on the unit circle, see~\citet{krengel1985ergodic}.

\begin{myex} 
Any Markov operator and in particular the Perron-Frobnenius operator $\pfo$ is a contraction on $L^1(X,\mu)$, this follows directly from the definition of a Markov operator, compare Section~\ref{sec:markovop}. If $\mu$ is an invariant measure than $\pfo$ is a contraction on the Hilbert space $L^2(X,\mu)$, see \citet{lasota2013chaos}.  In this case the Koopman operator $\koo$ is a contraction on $L^2(X,\mu)$ as well\cite{eisner2015operator}.
\end{myex}

The following lemmata will allow a slick proof of the classical mean ergodic theorem due to von Neumann and of a related theorem that implies the convergence of the $\varepsilon$-committors.
\begin{mylem}
 Let $O \in \mathcal{B}[H]$ be a contraction on a real or complex Hilbert space and $v \in H$. Then 
 \begin{equation}
  v=Ov \quad \Leftrightarrow \quad v = O^*v
  \end{equation}
\end{mylem}

\begin{proof}
If for some $v \in H : \lVert v \rVert^2 = \langle v, Ov \rangle$, then $\langle v, Ov \rangle$ is real and $\langle v, Ov \rangle = \langle Ov, v \rangle$ by symmetry of the scalar product. Then we get
\begin{align}
 \lVert Ov - v \rVert^2 & =  \langle Ov-v, Ov-v \rangle = \lVert Ov \rVert^2 + \lVert v \rVert^2 - 2 \langle v,Ov\rangle \nonumber \\
 & \leq  2 \lVert v \rVert^2  - 2 \lVert v \rVert^2 = 0,
\end{align}
where we used that $O$ is a contraction in the last line. Thus $v=Ov$ is equivalent to $ \lVert v \rVert^2 = \langle v,Ov \rangle = \langle O^*v,v \rangle$. Since $O^*$ is a contraction as well by Lemma~\ref{lem:adjnorm} applying the equivalence to $O^*$ yields the identity $v = O^*v$.
 \end{proof}
 
We will often use the subspace ${\fix(O)}\subset H$ of $O$-invariant vectors
\begin{equation}
{\fix(O)}:= \{ v \in H \mid Ov = v \}, 
\end{equation}
Obviously, ${\fix(O)}$ consists of the eigenvectors with eigenvalue 1 and is closed. 

A vector $u$ is called \emph{orthogonal} to a subspace $V \subseteq H$, if  $\langle u, v \rangle = 0 \quad \forall v \in V$. In this case we write $u \perp V$. The \emph{orthogonal complement} $V^\perp$ of a subspace $V$ is the set of all vectors $u$ that are orthogonal to $V$.

\begin{mylem}
 Let $O \in \mathcal{B}[H]$ be a contraction on a Hilbert space H. 
 Then the orthogonal complement ${\fix(O)}^\perp$ of ${\fix(O)}$ is the closure of the subspace $N$ spanned by $\{v-Ov \mid v \in H\}$.
 \label{lem:Tinvariant}
\end{mylem}

\begin{proof}
\begin{align}
  u \perp N & \Leftrightarrow  \langle u, (O-I)v \rangle = 0 \; \forall  v\in H \nonumber \\
  & \Leftrightarrow \langle O^*u - u, v \rangle = 0 \; \forall v \in H  \nonumber \\
  & \Leftrightarrow  O^*u = u \Leftrightarrow Ou = u \Leftrightarrow u \in {\fix(O)}.
\end{align}
Thus $N$ is orthogonal to ${\fix(O)}$. Since $\fix(O)^\perp$ is closed and contains $N$ it contains $\overbar{N}$ as well. Since a closed, linear subspace of a Hilbert space and its closure have the same orthogonal complement, we have that $\overbar{N}^\perp = {\fix(O)}$ and ${\fix(O)}^\perp = \left(\overbar{N}^\perp\right)^\perp = \overbar{N}$.
\end{proof}

A \emph{projection} is a linear map $P:H \to H$, such that $P^2 = P$. 
It induces a decomposition of $H = \ker P \oplus \im P$ into a direct sum of its kernel and its image. If its kernel and image are orthogonal onto each other, then $P$ is called \emph{orthogonal projection}. The projection operator onto $\ker P$ is  $Q:=I-P$ and it is easy to see that $QP =PQ = 0$. Conversely, if $H $ can be written as a direct sum of closed subspaces $U$ and $V$, then every element $h \in H = U \oplus V$ can be written as $h=u+v$ with $u \in U$ and $v \in V$. The map $P_U$ defined by $P_U h = u$, satisfies $P_U^2 = P_U$ and is called the projection onto U along V.

We are now well prepared to study the convergence of averages of powers of the operator. If O is a contraction we define $S_Nv := S_N[O]v = \frac{1}{N} \sum_{k=0}^{N-1} O^kv$, the so called  \emph{Ces\`{a}ro averages} or \emph{ergodic means}. Note the close connection to the expected mean sojourn times, that were introduced before.

\begin{mythm} (von Neumann mean ergodic theorem)

Let $O \in \mathcal{B}[H]$ be a contraction on a Hilbert space H. Then for every $v \in H$

\begin{equation}
 \lim_{N\to\infty} S_N[O]v = P_{\fix(O)}v,
\end{equation}
where $P_{\fix(O)}: H \to {\fix(O)}$ is the orthogonal projection onto the subspace ${\fix(O)}$.
\end{mythm}

\begin{proof} The argument is similar to the next proof, see also~\citet{krengel1985ergodic}, Thm. 1.4. 
\end{proof}
For a contraction O define the geometric averages $C_\varepsilon v := C_\varepsilon (O) v := \varepsilon \sum_{k=0}^\infty (1-\varepsilon)^k O^k v$. 
Since $\lVert O \rVert \leq 1$, it is a direct consequence of the summability of the geometric series, that for any $\varepsilon \in (0,1]$ the operator norm of $C_\varepsilon$ is bounded by 1 and that $C_\varepsilon$ is linear on $H$. 
Furthermore we have the identity 
\begin{equation}\label{eq:commind}
C_\varepsilon v = (1-\varepsilon)C_\varepsilon ( O v ) + \varepsilon v.
\end{equation}

\begin{mythm} (geometric mean ergodic theorem)

Let $O \in \mathcal{B}[H]$ be a contraction on a Hilbert space H. Then for every $v \in H$
 \begin{equation}
  \lim_{\varepsilon\to 0} C_\varepsilon[O]v = P_{\fix(O)} v,
 \end{equation}
 where $P_{\fix(O)}: H \to {\fix(O)}$ is the orthogonal projection onto the subspace ${\fix(O)}$.
 \label{thm:geomeanergo}
\end{mythm}

\begin{proof}
We see immediately that $C_\varepsilon v = P_{\fix(O)}v= v$ for all $v \in {\fix(O)}$. 

Let now $u = (O-I)v$ for some $v \in H$ then
 \begin{align}
  C_\varepsilon u & = C_\varepsilon (O v) - C_\varepsilon v =   C_\varepsilon (O v) - (1-\varepsilon)C_\varepsilon (O v) - \varepsilon v \nonumber \\ 
  & = \varepsilon ( C_\varepsilon (O v) - v)
 \end{align}
 Estimating the norm of the last term we get,
 \begin{equation}
  \varepsilon \lVert C_\varepsilon (O v) - v \rVert \leq \varepsilon (\lVert O v \rVert + \lVert v \rVert) = 2 \, \varepsilon \lVert v \rVert,
 \end{equation}

and this converges to 0 for $\varepsilon \to 0$. Now let $u$ be in the closure of $N:=(O-I)H$, then there is a sequence $(u_k)_{k \in \mathbb{N}} \in N$ that converges to $u$, such that $u_k = {(O-I)v_k}$ for some $v_k \in H$. Then for every $\delta > 0$, there is $K(\delta) \in \mathbb{N}$ such that $\lVert u - u_{K(\delta)} \rVert < \delta$ and hence
\begin{align}
 \lim_{\varepsilon \to 0} \lVert C_\varepsilon u \rVert 
 & \leq \lim_{\varepsilon \to 0} \left( \lVert C_\varepsilon (u - u_{K(\delta)}) \rVert + \lVert C_\varepsilon u_{K(\delta)} \rVert \right) \nonumber \\
 & \leq \lVert u - u_{K(\delta)} \rVert + \lim_{\varepsilon \to 0} 2 \, \varepsilon \lVert v_{K(\delta)} \rVert \nonumber \\
 & < \delta.
\end{align}
Since this inequality holds for all $\delta > 0$ we conclude that $ \lim_{\varepsilon \to 0} \lVert C_\varepsilon u \rVert = 0$ on the closure of $N$, which is equal to ${\fix(O)}^\perp$ by Lemma~\ref{lem:Tinvariant}.

If $\fix(O)$ is a closed, linear subspace, it is a well-known theorem that $ {H = {\fix(O)} \oplus {\fix(O)}^\perp}$. Then we can write any $u \in H$ as $u = v + w$ with $v\in {\fix(O)}, w \in {\fix(O)}^\perp$ and hence $\lim_{\varepsilon \to 0} C_\varepsilon u = \lim_{\varepsilon \to 0} ( C_\varepsilon v +  C_\varepsilon w )= v = P_{\fix(O)}u$ and thus $P_{\fix(O)}$ is an orthogonal projection.
\end{proof}

\begin{myrem}
 According to the theorem the convergence of $C_\varepsilon[O]$ to $P_{\fix(O)}$ is pointwise. If $H=\mathbb{R}^n$ this implies uniform convergence. For simplicity let $\lVert . \rVert$ denote the norm induced by the standard scalar product and $e_i$ the standard basis vectors. Denote $D_\varepsilon := C_\varepsilon[O] - P_{\fix(O)}$. Then
 \begin{align}
  \lVert D_\varepsilon \rVert = \sup_{\lVert x \rVert = 1} \lVert D_\varepsilon x \rVert & \leq \sup_{\lVert x \rVert = 1} \sum_{i=1}^N \lvert x_i \rvert \lVert D_\varepsilon e_i \rVert \nonumber \\ 
  & \leq N \cdot \max_{i=1,\dots,N} \lVert D_\varepsilon e_i \rVert \nonumber \\
  & \leq N \cdot \max_{i=1,\dots,N} 2 \, \varepsilon \lVert  f_i \rVert \to 0,
 \end{align}
 for $\varepsilon \to 0$, where $f_i = 0$ if $e_i \in \fix(O)$, or else $f_i \in H $ is such that $e_i = (O-I)f_i$. In particular the convergence is uniform if $O$ is a contractive matrix.
\end{myrem}

\begin{myrem} We suppose that for compact, normal operators the convergence is uniform  as well. A proof via the spectral theorem\cite{kubrusly2012spectral} might be possible, is however beyond the scope of this work.
\end{myrem}

\subsection{Brief summary of spectral theory for Hilbert space operators}

In the next section we will prove the mean ergodic theorems for another class of operators, which are not necessarily contractions. The present section introduces some of the tools needed for the proof, most notably we establish a link between the spectral radius of an operator and the convergence of its powers (Corollary~\ref{cor:unistab}).

Let H be a complex Banach space and $O:H \to H$ a bounded, linear operator. The \emph{resolvent set} $\rho(O)$ of O is the set of all $\lambda \in \mathbb{C}$, such that the operator $\lambda I - O$ is invertible with a bounded, linear inverse. Its complement $\sigma(O) := \mathbb{C} \setminus \rho(O)$ is called the \emph{spectrum} of $O$. The spectrum can be split into disjoint parts, depending on the reason why the operator $\lambda I - O$ fails to be invertible.

The most important part for our purposes is the \emph{point spectrum}
\begin{equation}
 \sigma_P(O) := \{ \lambda \in \mathbb{C} \mid \ker(\lambda I - O) \neq \{0\} \}.
\end{equation}
The other parts are called the \emph{continuous spectrum}
\begin{align}
 \sigma_C(O) :=  \{ & \lambda \in \mathbb{C} \mid \ker(\lambda I - O) = \{0\}, \ \im(\lambda I - O) \neq H \nonumber \\
  & \mathrm{ and } \ \overbar{\im(\lambda I -O) } = H \},
\end{align}
and the \emph{residual spectrum}
\begin{equation}
  \sigma_C(O) := \{ \lambda \in \mathbb{C} \mid \ker(\lambda I - O) = \{0\}, \ \overbar{\im(\lambda I -O) } \neq H \}.
\end{equation}
Every $\lambda \in \sigma_P(O)$ is called an eigenvalue of $O$ and the corresponding eigenvectors are the elements of ${\ker (\lambda I - O)}$, which is the eigenspace of $O$ at eigenvalue $\lambda$. 

An important class of operators for which the structure of the spectrum is particularly simple and well understood are compact operators.
An operator $O$ is called \emph{compact}, if $OA$ is relatively compact for every bounded subset $A \subset H$. 
\begin{myrem} (Matrices)
 If $H$ is finite-dimensional, then $O$ is compact. In particular every matrix is a compact operator. This is a consequence of the Heine-Borel Theorem, which states that in finite dimensional spaces a subset is compact, if and only if it is closed and bounded. 
\end{myrem}
\begin{myrem} (Compact Domain)
 If $H$ is compact then every map from $H$ to itself is compact. 
\end{myrem}

We will now state without proof a number of general results on compact operators $O \in \mathcal{B}[H]$. The proofs are omitted since they require advanced techniques that have little in common with the main subject of this paper. For details we refer to \citet{kubrusly2012spectral}.

The so-called \emph{Fredholm Alternative states} that the residual and continuous parts of the spectrum of a compact operator on a Hilbert space are either empty or $\{0\}$. In other words, the non-zero spectrum of a compact operator equals its point spectrum.
\begin{mythm} (Fredholm Alternative)
 Let $O:H\to H$ be a compact, bounded, linear operator, then
 \begin{equation}
   \sigma(O) \setminus \{0\} =\sigma_P(O) \setminus \{0\} 
 \end{equation}
 \label{thm:fredholm}
\end{mythm}
Furthermore the spectrum is a countable set and its only possible accumulation point is~0, see~\citet{kubrusly2012spectral}, Cor. 2.20. 

The \emph{spectral radius} of an operator O is defined as
\begin{equation}
 r(O) = \sup_{\lambda \in \sigma(O)} \lvert \lambda \rvert.
\end{equation}
The \emph{Gelfand-Beurling formula} establishes a connection between the spectral radius of $O$ and the norm of its powers $\lVert O^n \rVert$, it states that
\begin{equation}
 r(O) = \lim_{n \to \infty} \lVert O^n \rVert ^ {1/n}. \label{eq:gelfand}
\end{equation}
A proof can be found in~\citet{kubrusly2012spectral}, Thm. 2.10. This formula allows to prove that the power of an operator converges uniformly to 0 if and only if its spectral radius is strictly smaller than 1.
\begin{mycor}
 Let $O$ be a bounded, linear operator on a complex Banach space, then
 \begin{equation}
  r(O) < 1 \quad \Leftrightarrow \quad \lim_{n \to \infty} \lVert O^n \rVert = 0.
 \end{equation}
 \label{cor:unistab}
\end{mycor}

\subsection{Ergodic theorems for a class of decomposable Hilbert space operators \label{sec:ergodeco}}

We are now ready to prove the mean ergodic theorems for Hilbert space operators $O$ that admit a decomposition into the sum of a unitary operator and an operator with spectral radius smaller than 1. An imporant class of such operators are stochastic matrices, as we shall see in the next section.
\begin{mythm}
 Let $O \in \mathcal{B}[H]$ on a Hilbert space H. Assume $H = \mathcal{U} \oplus \mathcal{V}$ is the direct sum of $O$-invariant closed subspaces $\mathcal{U}$ and $\mathcal{V}$, such that $U:= OP_\mathcal{U}$ is unitary and $V := OP_\mathcal{V}$ has spectral radius $r(V)<1$. Then
 
 \begin{equation}
  \lim_{N \to \infty} S_N[O]v = P_{\fix(O)}v = \lim_{\varepsilon \to 0} C_\varepsilon[O]v \quad \forall v \in H,
 \end{equation}
 where ${\fix(O)} = \ker(I - O)$ is the subspace of $O$-invariant vectors and $P_{\fix(O)}$ is an orthogonal projection.

\end{mythm}

\begin{proof}
 Since the subspaces $\mathcal{U}$ and $\mathcal{V}$ are O-invariant, we have that $OP_\mathcal{U} = P_\mathcal{U}O$ and  $OP_\mathcal{V} = P_\mathcal{V}O$. This implies $UV = VU = 0$ and further $O^k = (U+V)^k = U^k + V^k$. Hence the averages $S_N[O] = S_N(U) + S_N(V)$ and  $C_\varepsilon[O]= C_\varepsilon(U) + C_\varepsilon(V)$ split into two distinct terms. The mean ergodic theorems imply that 
 \begin{equation}
  \lim_{N\to\infty} S_N(U)v = P_{\ker(I-U)} = \lim_{\varepsilon \to 0} C_\varepsilon(U)v.
 \end{equation}
We show now that on the other hand $\lim_{N\to\infty} S_N(V)v = 0 = \lim_{\varepsilon \to 0} C_\varepsilon(V)v$.
 
By assumption $r(V) < 1$ and hence Corollary~\ref{cor:unistab} implies $\lim_{k \to \infty} \lVert V^k \rVert = 0$. It follows that for all $\delta > 0$ there exists $M \in \mathbb{N}$, such that $\lVert V^k \rVert < \delta $ for all $k \geq M$. Then
 \begin{align}
  \lim_{N \to \infty} \lVert S_N(V) \rVert & \leq  \lim_{N \to \infty} \frac{1}{N}\sum_{k=0}^{M-1} \lVert O^k \rVert + \lim_{N \to \infty} \frac{1}{N}\sum_{k=M}^{N-1} \lVert O^k \rVert  \nonumber \\  
  & \leq  \lim_{N \to \infty} (\frac{M-1}{N} + \delta \frac{N-M-1}{N}) = \delta,
 \end{align}
and since this holds for all $\delta > 0$ we have that $\lim_{N \to \infty} S_N(V) = 0$.
By an analogous argument one establishes $\lim_{\varepsilon \to 0} C_\varepsilon(V) = 0$.
It remains to show that $\ker(I-U) = \ker(I-O)$. By assumption for every $x \in H = \mathcal{U} \oplus \mathcal{V}$ there exist $u \in \mathcal{U}$ and $v \in \mathcal{V}$ such that $x = u+v$. Assume that $Ox=x$, then 
\begin{equation}
 Ox = x \Leftrightarrow (U+V)(u+v) = (u+v) \Leftrightarrow Uu + Vv = u + v,
\end{equation}
and since $Uu \in \mathcal{U}$ and $Vv \in \mathcal{V}$ this is equivalent to
\begin{equation}
 Uu = u \quad \mathrm{and} \quad Vv = v.
\end{equation}
Since V has spectral radius smaller than 1 it cannot have any fixed points except $0$ and hence $v = 0$, which implies that $x = u$. Hence $Ox=x$ if and only if $Ux=x$ or equivalently $\ker(I-U) = \ker(I-O)$.
\end{proof}

\begin{myrem}
 The projection operator in the theorem is the orthogonal projection onto $\ker(I-U)$ along $\im(I-U)$, however in general $\ker(I-U)$ need not be orthogonal onto the subspace $\mathcal{V}$ from the theorem. 
\end{myrem}


\subsection{Ergodic theorems for stochastic matrices \label{sec:ergomat}}

Having established the mean ergodic theorems for fairly general Hilbert space operators, we now turn to most important case for the applications we have in mind, which are stochastic matrices. In this section we will introduce some basic properties of stochastic matrices and then show that they admit a decomposition like the one described in Section~\ref{sec:ergodeco}. We conclude this section by proving the mean ergodic theorems for stochastic matrices.

From now on let $M \in \mathbb{R}^{n\times n}$ be a stochastic matrix. A right eigenvector of $M$ at eigenvalue $\lambda$ is a solution to the equation $Mv = \lambda v$ and similarly a left eigenvector solves $v^T M = \lambda v^T$. A left eigenvector of $M$ is a right eigenvector of $M^T$ and vice versa.  Since $M$ and $M^T$ share the same characteristic polynomial, their spectra and in particular their spectral radii are the same.

\begin{mylem}
 $M$ has spectral radius $r(M)=1$.
\end{mylem}

\begin{mylem} \label{lem:stochasticpft}
 M has at least one left eigenvector at eigenvalue 1, which is a probability distribution vector.
\end{mylem}

An eigenvalue of a matrix is called \emph{semisimple}, if its geometric multiplicity equals its algebraic multiplicity or equivalently if the corresponding eigenspace admits an orthogonal eigenbasis (over $\mathbb{C}$). If all eigenvalues of a matrix are semisimple, it is diagonalizable\cite{meyer2000matrix}.

\begin{mythm}
If $M$ is a stochastic matrix, then all eigenvalues $\lambda_i$ with $\lvert \lambda_i \rvert = 1$ are semisimple.
\end{mythm}

\begin{mythm}
 Let $M \in \mathbb{R}^{n \times n}$ be a stochastic matrix and let $Q \in \mathbb{R}^{n \times n}$ be invertible, such that $J=Q^{-1}MQ$ is the Jordan normal form of $M$, then
 \begin{equation}
  \lim_{N \to \infty} S_N[M]  = \lim_{\varepsilon \to 0} C_\varepsilon[M] = P_{\fix(M)},
 \end{equation}
 where $P_{\fix(M)}$ is a projection onto ${\fix(M)} = \{ v \in \mathbb{R}^n \mid Mv=v\}$ given by
 \begin{equation}
  P_{\fix(M)} = Q^{-1} P_{\fix(J)} Q,
 \end{equation}
and $P_{\fix(J)}$ is an orthogonal projection.
\end{mythm}

\begin{proof}
 We assume that the reader is familiar with the Jordan normal form, a good reference is~\citet{meyer2000matrix}.  For the Jordan normal form $J$ of $M$ it is obvious that eigenspaces corresponding to different eigenvalues are orthogonal onto each other. As stated above  the spectral radius of a stochastic matrix is 1 and all eigenvalues on the unit circle are semisimple. Hence $J = U + V$ can be decomposed into a unitary part and a part with spectral radius smaller 1, such that $UV=VU=0$. Since matrices are bounded, linear operators on finite-dimensional spaces, pointwise convergence implies uniform convergence. Thus applying Theorem~\ref{thm:operatorlimits} to $J$ we have 
 \begin{equation}
  \lim_{N \to \infty} S_N(J) = P_{\fix(J)}.
 \end{equation}
Using the identity $M^2 = (Q^{-1} J Q)^2 = Q^{-1} J QQ^{-1} J Q = Q^{-1} J^2 Q$ it is easy to see that
\begin{equation}
 S_N[M] = S_N (Q^{-1} J Q) = Q^{-1} S_N(J) Q,
\end{equation}
and by taking limits we get
\begin{equation}
 \lim_{N\to\infty} S_N[M] = \lim_{N\to\infty} Q^{-1} S_N(J) Q = Q^{-1} P_{\fix(J)} Q.
\end{equation}
It remains to show that $\im (Q^{-1} P_{\fix(J)} Q) = \fix(M)$, which holds since
\begin{align}
   x^T \in \im (Q^{-1} P_{\fix(J)} Q) &  \Leftrightarrow x^T Q^{-1} P_{\fix(J)} Q  \neq 0    \Leftrightarrow x^T Q^{-1} P_{\fix(J)}  \neq 0 \nonumber \\ 
   &   \Leftrightarrow x^T Q^{-1}J=x^T Q^{-1}    \Leftrightarrow  x^T Q^{-1}JQx = x^T  \nonumber \\ 
   & \Leftrightarrow x^T M = x^T. 
\end{align}
The argument for the limit $\lim_{\varepsilon \to 0} C_\varepsilon[M]$ is analogous.
\end{proof}

\begin{myrem}
As for any closed subspace of a Hilbert space, there exists an orthogonal projection onto $\fix(M)$. 
However the projection $P_{\fix(M)}$ that we get from the theorem is in general not orthogonal.
This happens to be so, since stochasticity of a matrix is a property that only holds with respect to a certain basis of $\mathbb{R}^n$, namely the standard normal basis, where every basis vector corresponds to a certain state of the associated Markov chain. 
On the other hand the spectrum of a linear operator is independent of the basis and the theorem is mainly a consequence of the spectral properties of $M$.
We obtain $P_{\fix(M)}$ by switching to a suitable basis, such that $M$ has Jordan normal form, which allows us to apply the results of the previous section. The resulting projection matrix $P_{\fix(M)}$ is sometimes called spectral projection of $M$ at eigenvalue 1.
\citet{meyer2000matrix} gives an explicit characterization of $P_{\fix(M)}$ in terms of sub-matrices of $M$. 
\end{myrem}

Above we saw that the right fixed points of the transition matrix associated to a Perron-Frobenius are expected to be almost constant on the basins, compare Section~\ref{sec:ctsop} and Example~\ref{ex:emstulamfix}. For general Markov chains \citet{deuflhard2005robust} give some intuition on the structure of the right 1-eigenvectors.
In the ideal case of a Markov chain consisting of several uncoupled sub-chains, the right 1-eigenvectors will be constant on the irreducible components. If a Markov chain has several metastable states and transitions between these states are rare events, then it can be thought of as small perturbations of such an ideal chain. \citet{deuflhard2005robust} show that for nearly uncoupled chains the perturbed 1-eigenvectors have eigenvalues close to 1 and are almost constant on the metastable components. They exploit this fact to approximate metastable states. However in the presence of long transients this constant level pattern is in general not preserved and more complex algorithms are required\cite{roblitz2013fuzzy, weber2015g}.

\section{Difference estimates for metastable states ~\label{sec:differenceestimates}}

This section does not intend to provide thorough treatment of committor functions and EMS times in the presence of metastability, but rather aims to develop our intuition of the behaviour that is to be expected. Metastability in Markov chains is an extensive area of research and various related notions of metastable states and almost-invariant sets have been proposed\cite{bovier2002metastability, froyland2009almost, roblitz2013fuzzy}. For our purposes it will suffice to work with the straightforward idea that (left) eigenvectors with real eigenvalues close to 1 characterize metastable states. Note that the eigenvalues need to be real in order to avoid oscillations. 

Let $M$ be the transition matrix of a Markov chain and $v$ an eigenvector of $M$ at a real eigenvalue $\lambda \approx 1$, normalized such that $\lVert v \rVert = 1$. Then $\lVert v^TM - v^T \rVert = (1-\lambda) \lVert v \rVert \approx 0$ and the systems state described by $v$ hardly changes during one iteration of the system and may be considered as invariant on short time-scales.

If we consider $C_\varepsilon$ and $S_N$ as expected values of time averages along trajectories, then $C_\varepsilon$ corresponds to averaging with respect to a geometric distribution with parameter $\varepsilon$, while $S_N$ corresponds to averaging with respect to the equidistribution on $\{0,\dots,N-1\}$ (cf. Section~\ref{sec:gbs}). In order to compare both averages we align the expected values of these distribution by choosing $\varepsilon = \frac{2}{N+1}$ and write in abuse of notation  $ C_\varepsilon = C_{\frac{2}{N+1}} =: C_N$.

Then $s_N(v) = S_N[\lambda]v$,  $q_N(m)= C_N[\lambda]v$ and for the difference term $h(\lambda,N) := \lVert s_N(v) - q_N(v) \rVert$ we have
\begin{align}
h(\lambda,N) & =  \lVert  S_N[\lambda]v - C_N[\lambda]v \rVert  \nonumber  = \lvert  S_N[\lambda] - C_N[\lambda]\rvert \cdot \lVert v \rVert \nonumber  \\
 & =  \lvert \frac{1}{N} \cdot \frac{1-\lambda^N}{1-\lambda} - \frac{2}{N+1} \cdot \frac{1}{1-(1-\frac{2}{N+1})\lambda} \rvert .
 \label{eq:exponentialdifference}
\end{align}
Elementary analytic arguments reveal that the difference vanishes for fixed $\lambda$ and $N \to \infty$ as expected.

Figure~\ref{fig:differenceterms} shows the difference terms for varied $N$ and several eigenvalues $\lambda$, possibly far from 1. We see that for sufficiently large $N$ the difference between EMS time and $\varepsilon$-committor vanishes. Further we observe that on small time-scales $N$ such that the metastable state has not significantly decayed yet, that is $\lVert v^TM^N-v^T \rVert = 1-\lambda^N \approx 0$, the difference term is small as well. Among our example values this effect is most significant for $\lambda=0.999$, that is for a slowly decaying metastable states. 

\begin{figure}
\centering
\includegraphics[width=0.8\columnwidth]{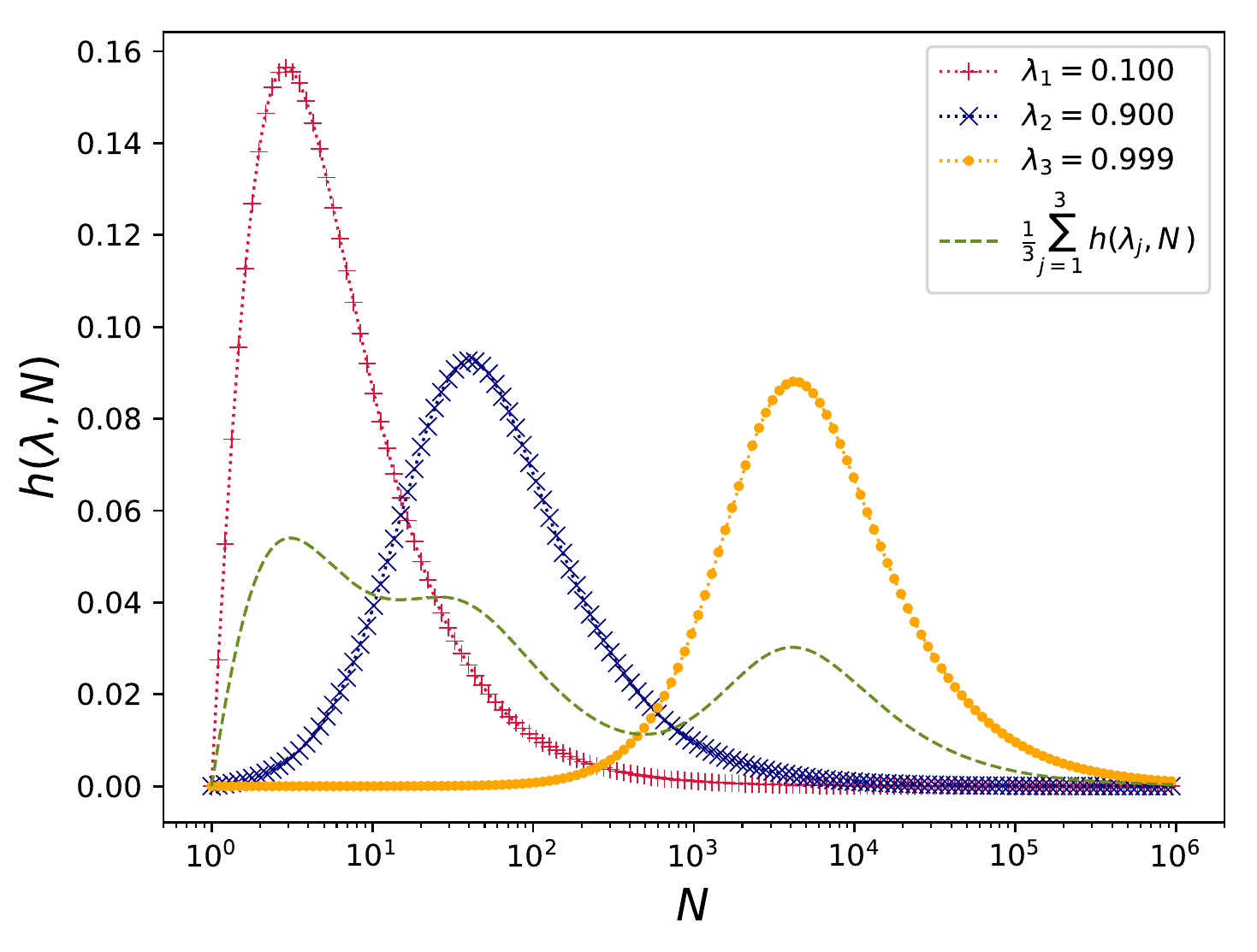}
 \caption{Error term $ h(\lambda,N)$ as described by Eq.~\eqref{eq:exponentialdifference} for different values of $\lambda$ and for $\varepsilon=\frac{2}{N+1}$. The green dashed line shows a difference term for a more complex metastable state given by a linear combination of three eigenvectors. In this case $\beta=\frac{1}{3}$ and $\alpha_i=1$ for all $i=1,2,3$.}
 \label{fig:differenceterms}
\end{figure}

If we extend our notion of a metastable state and allow $m = \beta \left( \sum_{i=1}^L \alpha_i v_i \right)$, where $v_i$ is an eigenvector of $M$ at real eigenvalue $\lambda_i$ with $\lVert v_i \rVert=  1$ for all $i$, and $\alpha_1,\dots,\alpha_L,\beta \in [0,1]$, such that $\lVert m \rVert = 1$, then the difference term $\lVert S_N[M]m - C_N[M]m\rVert$ is bounded by
\begin{align}
  &\lVert \beta \sum_{i=1}^L \alpha_i v_i^T S_N[M]  - \beta \sum_{i=1}^L \alpha_i v_i^T C_N[M]     \rVert \nonumber \\
  &\leq \beta \sum_{i=1}^L \alpha_i \lVert v_i^T S_N[\lambda_i]  - v_i^T C_N[\lambda_i]  \rVert \nonumber \\
  & =\beta \sum_{i=1}^L \alpha_i h(\lambda_i,N),
\end{align}
that is by the weighted sum of the difference terms corresponding to the different eigenvectors. In Figure~\ref{fig:differenceterms} an example of such a combined term is plotted along the individual terms corresponding to the distinct eigenvalues. We see that the difference term for the combined metastable set decays faster than the one corresponding to the largest eigenvalue. Thus it suffices to choose $N$ large enough, such that $h(\lambda,N)$ is small for the largest $\lambda$ in order to ensure that the difference term of a metastable state is small.

\begin{myrem} (Right eigenvectors)
 Starting out with right eigenvectors we analogously obtain the same difference terms. Since $\mathbbm{1}_X$ is a right eigenvector, we can always find non-negative linear combinations of right eigenvectors, that is metastable probability densities. These are related to the systems basin structure, however their precise significance for the dynamics is less clear.
\end{myrem}

\begin{myrem} (Complex eigenvalues)

A generic transition matrix $M$ is diagonalizable over $\mathbb{C}$ with possibly complex eigenvalues and eigenvectors. General linear combination of the kind ${m = \beta \left( \sum_{i=1}^L \alpha_i v_i \right)}$, with $\alpha$ and $\beta$ as above and with complex eigenvectors $v_i$ such that $\lVert v_i \rVert_\infty = 1$, may no longer be interpretable as  metastable states due to osicillations, however the difference terms $h(\lambda,N)$ remain valid. We observe numerically that $h(\lambda,N)$ takes larger maximal values for $\lambda \in \mathbb{C}$. On the other hand simulations show that $h(\lambda,N)$ converges faster to zero for complex $\lambda$ than for real $\lambda$ with the same absolute value. Hence in order to choose $N$ large enough for the difference term of such a generalized state $m$ to vanish we suggest to consider $h(\lvert \lambda \rvert, N)$ for the eigenvalue $\lambda$ of largest absolute value, compare Figure~\ref{fig:complexdifference}.
\end{myrem}

\begin{figure}
\centering
\includegraphics[width=0.8\columnwidth]{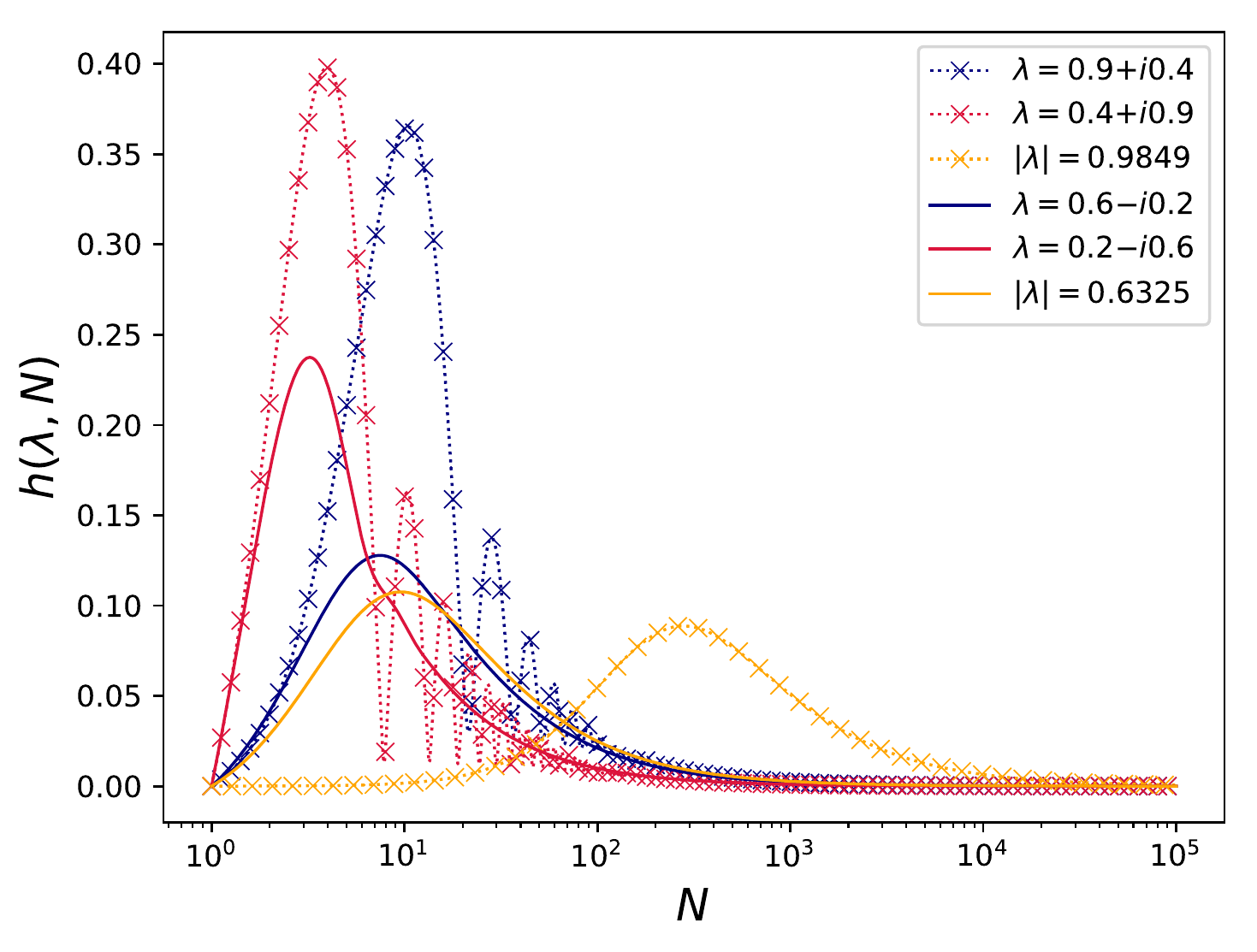}
\caption{$h(\lambda,N)$ for various complex $\lambda$ is shown in red and blue and for comparison $h(\lvert \lambda \rvert, N)$ is shown in yellow. We see that for complex $\lambda$ the difference term takes higher maximal values and converges faster to zero than for their absolute values.}
 \label{fig:complexdifference}
\end{figure}

\begin{myrem} (Detecting metastable states)

Let $x,y \in \mathcal{P}:=\{ \rho \in \mathbb{R}^n \mid \lVert \rho \rVert = 1, \rho \geq 0 \}$ be probability distribution vectors. The expression 
 \begin{equation} \label{eq:metamin}
  x^T C_\varepsilon[M] y,
 \end{equation}
may be understood as the probability that the process with $\varepsilon$-absorption terminates in state $y$ given it started in state $x$. Intuitively, if a set is metastable at time-scale $\varepsilon^{-1}$, then $x^T C_\varepsilon[M] x$ should have a local minimum at $x\in\mathcal{P}$. A promising question for future research might be to study such local minima and their connection to metastability more closely.

\end{myrem}

\clearpage

\end{document}